\documentclass[reqno]{amsart}
\usepackage{amsmath,amsthm,amssymb,amsfonts,color}
\usepackage{hyperref,comment}
\usepackage{esint}
\usepackage[latin1]{inputenc}

\usepackage[shortlabels]{enumitem}

\usepackage{todonotes}

\newcommand\bR{\mathbb{R}}

\newcommand\kC{C_{\text{kin}}}

\newcommand\cR{\mathcal{R}}

\newcommand\tQ{\widetilde Q}

\newcommand\llocxv{L_{2; \,  \text{loc}, x, v}  }

\newcommand\slocxv{S_{2; \,  \text{loc}, x, v}  }
\renewcommand{\epsilon}{\varepsilon}

\numberwithin{equation}{section}

\newcommand\cbrk{\text{$]$\kern-.15em$]$}}
\newcommand\opar{
\text{\,\raise.2ex\hbox{${\scriptstyle |}$}\kern-.34em$($}}

 \theoremstyle{definition}

\newtheorem{theorem}{Theorem}[section]
\newtheorem{lemma}[theorem]{Lemma}
\newtheorem{proposition}[theorem]{Proposition}
\newtheorem{corollary}[theorem]{Corollary}

\newtheorem{definition}{Definition}[section]

\theoremstyle{remark}
\newtheorem{remark}[theorem]{Remark}

\newtheorem{assumption}[theorem]{Assumption}

\newcommand{\nlimsup}{\operatornamewithlimits{\overline{lim}}}

\begin{document}

\title[Global Schauder estimates for KFP equations]{Global Schauder estimates for kinetic Kolmogorov-Fokker-Planck equations}

\author[H. Dong]{Hongjie Dong}
\address[H. Dong]{Division of Applied Mathematics, Brown University, 182 George Street, Providence, RI 02912, USA}
\email{Hongjie\_Dong@brown.edu }
\thanks{H. Dong was partially supported by the Simons Foundation, grant no. 709545, a Simons fellowship, grant no. 007638, the NSF under agreement DMS-2055244, and the Charles Simonyi Endowment at the Institute for Advanced Study.}

\author[T. Yastrzhembskiy]{Timur Yastrzhembskiy}
\address[T. Yastrzhembskiy]{Division of Applied Mathematics, Brown University, 182 George Street, Providence, RI 02912, USA}
\email{Timur\_Yastrzhembskiy@brown.edu}


\begin{abstract}
We  present  global  Schauder  type  estimates in all variables and unique solvability results in kinetic H\"older  spaces for kinetic Kolmogorov-Fokker-Planck (KFP) equations.
The leading coefficients are H\"older continuous in the $x, v$ variables and are merely measurable in the temporal variable.
Our proof is inspired by Campanato's approach to Schauder  estimates and does not rely on the estimates of the fundamental solution of the KFP operator.
\end{abstract}

\maketitle

\tableofcontents

\section{Introduction and main result}
                            \label{section 1}
Let $d \ge 1$, $x \in \bR^d$ be the spatial variable, $v \in \bR^d$ be the velocity variable, and denote $z = (t, x, v)$.
Throughout the paper, $T\in (-\infty, \infty]$, and $\bR^{1+2d}_T: = (-\infty, T) \times \bR^{2d}$.
The goal of this article is to establish a Schauder type estimate for the KFP equation
\begin{equation}
                \label{1.1}
   P u + b \cdot D_{v} u  + (c+\lambda^2) u = f,
\end{equation}
where
\begin{equation}
                \label{1.2}
    P: =  \partial_t u - v \cdot D_x u - a^{i j} (z) D_{v_i v_j} u.
\end{equation}
The above equation appears in kinetic theory, theory of diffusion processes, and mathematical finance (see \cite{P_05} and the references therein). In particular, \eqref{1.1} with $-v \cdot D_x u$ replaced with $v \cdot D_x u$ can be viewed as a linearization of the Landau  equation (LE) (see \cite{AV_04}), an important model of  weakly coupled plasma.  We also mention that $P$ is the infinitesimal generator of the Langevin diffusion process (see \cite{P_14}), so that the time-reversed version of \eqref{1.1} can be  viewed as a backward Kolmogorov equation for the Langevin process.

It is a fundamental problem to establish the maximal regularity for the KFP equation in various functional spaces such as H\"older spaces (see  \cite{BB_22}, \cite{CHM_21}, \cite{DFP_06}, \cite{HS_20},  \cite{HW_22}, \cite{IM_21}, \cite{L_97}, \cite{M_97}, \cite{PRS_22}) and $L_p$ spaces (see \cite{BCLP_13}, \cite{CZ_19}, \cite{DY_21a} - \cite{DY_21b} and the references therein) that is analogous to the theory developed for nondegenerate equations (see, for example, \cite{F_64}, \cite{Kr_96}, \cite{Kr_08}). Such results play a crucial role in the studies of the  conditional regularity of the LE (see \cite{HS_20}) and boundary value problem for the LE with the specular reflection boundary condition (see \cite{DGO_20} - \cite{DGY_21}).

The purpose of the present paper is threefold. First, we establish \textit{global}  Schauder estimates in $t, x, v$ variables for Eq. \eqref{1.1} with \textit{time irregular} leading coefficients (see Theorem \ref{theorem 1.5}). This result  appears to be new (see the discussion in Sections \ref{section 1.4} - \ref{section 1.7}).
Second, we show how the constant on the right-hand side of the a priori estimate \eqref{eq1.5.2} depends on the lower eigenvalue bound $\delta$ (see Assumption \ref{assumption 1.1}). This is relevant to the linearization of the LE near the global Maxwellian  because for such an equation, one has
$$
	N_1 |v|^{-3} \delta_{i j} < 	a^{i j} (z) \le N_2 |v|^{-1} \delta_{i j}.
$$
Hence, localizing to the velocity shell $|v| \sim 2^n$, we obtain the equation of type \eqref{1.1} with $\delta \sim 2^{-3n}$ (see the details in \cite{DGO_20} - \cite{DGY_21}).  Third, we  prove the a priori estimates without using the  fundamental solution of the KFP equation.
Our method is inspired by Campanato's approach. See the details in Section \ref{section 1.3}.  This article is a part of the present authors' program to develop the maximal regularity results for the KFP equation  with rough leading coefficients via kernel free approach (see \cite{DY_21a} - \cite{DY_21b}).

Before we state the main result and review the relevant literature, we introduce some notation.

\subsection{Notation}
                \label{section 1.1}
                 In this section, $\alpha \in (0, 1]$ is a number, and $G \subset \bR^{1+2d}$ is an open set.

\textit{The usual H\"older space.}
For an open set $\Omega \subset \bR^d$, by  $C^{\alpha} ( \Omega)$, we mean the usual H\"older space with the seminorm
$$
    [u]_{  C^{\alpha} ( \Omega) } := \sup_{ x, x' \in \Omega: x \neq x'} \frac{|u (x) - u (x')|}{|x-x'|^{\alpha}},
$$
and the norm
$$
     \|u\|_{  C^{\alpha} ( \Omega) } := \|u\|_{ L_{\infty} ( \Omega) } + [u]_{  C^{\alpha} (\Omega) }.
$$

\textit{Anisotropic H\"older spaces.}
For $\alpha \in (0, 1]$ and an open set $D \subset \bR^{2d}$, we denote
$$
    [u]_{C^{\alpha/3, \alpha}_{x, v} (D)}: = \sup_{(x_i, v_i) \in D: (x_1, v_1) \neq (x_2, v_2)} \frac{|u (x_1, v_1) - u (x_2, v_2)|}{(|x_1-x_2|^{1/3} + |v_1-v_2|)^{\alpha}}.
$$
Furthermore, for an open set of the form
\begin{equation}
\label{eq1.12}
    G = (t_0, t_1) \times D, \,  -\infty \le t_0 < t_1 \le \infty,
\end{equation}
 we set
\begin{align*}
& L_{\infty} C^{\alpha/3, \alpha}_{x, v} (G) : = L_{\infty} ((t_0, t_1), C^{\alpha/3, \alpha}_{x, v} (D)), \\
&
    [u]_{ L_{\infty} C^{\alpha/3, \alpha}_{x, v} (G)} = \text{ess sup}_{t \in (t_0, t_1)} [u (t, \cdot)]_{C^{\alpha/3, \alpha}_{x, v} (D)}, \\
    &  \|u\|_{ L_{\infty} C^{\alpha/3, \alpha}_{x, v} (G)} = \|u\|_{L_{\infty} (G) } +  [u]_{ L_{\infty} C^{\alpha/3, \alpha}_{x, v} (G)}.
\end{align*}
Furthermore, we say that $u \in \mathbb{C}^{2, \alpha} (G)$ if
$$
    u, D_v u, D^2_v u,  \partial_t u - v \cdot D_x u \in  L_{\infty} C^{\alpha/3, \alpha}_{x, v} (G).
$$
We stress  that  $\partial_t u$ and $v \cdot D_x u$ are understood in the sense of distributions.
The norm in this space is defined as
\begin{equation}
    \label{1.7}
    \|u\|_{\mathbb{C}^{2, \alpha} (G)} : =  \|u\| + \|D_v u\| + \|D^2_v u\| + \|\partial_t u - v \cdot D_x u\|,
\end{equation}
where $\|\cdot\| = \|\cdot\|_{_{L_{\infty} C^{\alpha/3, \alpha}_{x, v} (G)}}$.

\textit{Kinetic quasi-distance and kinetic H\"older spaces.}
We denote
\begin{equation}
                \label{1.3}
        \rho (z, z_0) =   \max\{ |t-t_0|^{1/2},  |x - x_0 + (t - t_0) v_0|^{1/3}, |v-v_0|\}.
\end{equation}
Note that $\rho$ satisfies all the properties of the quasi-distance except the symmetry.
By $\widehat \rho$ we denote a symmetrization of $\rho$  given by
\begin{equation}
                \label{1.4}
    \widehat \rho (z, z') = \rho (z, z')+\rho(z', z).
\end{equation}

We introduce the kinetic H\"older seminorm
\begin{equation}
                \label{1.5}
    [u]_{  \kC^{\alpha} (G) } := \sup_{ z, z' \in G: z \ne z'} \frac{|u (z) - u (z')|}{\rho^{\alpha} (z, z')}
\end{equation}
and the kinetic H\"older space
$$
    \kC^{\alpha} (G):=\{u \in L_{\infty} (G): [u]_{  \kC^{\alpha} (G) } < \infty\}
$$
equipped with the norm
$$
    \|u\|_{  \kC^{\alpha} (G) } = \|u\|_{ L_{\infty} (G) } + [u]_{  \kC^{\alpha} (G) }.
$$

Furthermore, for an open set of the form \eqref{eq1.12},  we define $\kC^{2, \alpha} (G)$ to be the Banach space of all $\kC^{\alpha} (G)$ functions $u$ such that  the norm
 \begin{equation*}
    \|u\|_{ \kC^{2, \alpha} (G) } := \|u\|_{ \kC^{ \alpha} (G) } + \|\partial_t  u - v \cdot D_x u\|_{ L_{\infty} C^{\alpha/3, \alpha}_{x, v} (G) } + [D^2_v u]_{ \kC^{\alpha} (G) }
\end{equation*}
is finite.

\begin{remark}
Due to Lemma \ref{lemma B.1} $(i)$, replacing $\rho (z, z')$ with $\widehat \rho (z, z')$ in \eqref{1.5} yields an equivalent space.
\end{remark}

\begin{remark}
\label{remark 1.5}
Our definition of the spaces $\kC^{\alpha}$ and $\kC^{2, \alpha}$ is similar to  those used in \cite{M_97} and \cite{IM_21}.
In particular, it follows from Remark 2.9 in \cite{IM_21} that the $C^{2+\alpha}_l$ seminorm  (see Definition 2.2 therein)  is equivalent to
$$
    [\partial_t u + v \cdot D_x u]_{ \widetilde C^{\alpha} (\bR^{1+2d}) } + [D^2_v u]_{ \widetilde C^{\alpha} (\bR^{1+2d}) },
$$
where
\begin{align}
\label{1.15}
  &  [f]_{ \widetilde C^{\alpha} (\bR^{1+2d}) }:= \sup_{ z, z' \in \bR^{1+2d}: z \ne z'} \frac{|u (z) - u (z')|}{ d_l (z, z')}, \\
  &  d_l (z, z') =  \max\{ |t-t'|^{1/2},  |x - x' - (t - t') v'|^{1/3}, |v-v'|\}\notag.
\end{align}
\end{remark}

\textbf{Convention.} By $N = N (\cdots)$ and $\theta = \theta (\cdots)$,  we denote  constants depending only on the parameters inside the
parentheses. These constants  might change from line to line. Sometimes, when it is clear what parameters $N$ and $\theta$  depend on, we omit  them.

\subsection{Main results}
                    \label{section 1.2}
\begin{assumption}
                \label{assumption 1.1}
The function $a = (a^{i j} (z), i, j = 1, \ldots, d)$ is measurable, and there exists  some $\delta \in (0, 1)$  such that
$$
    a^{i j} \xi_i \xi_j \ge \delta |\xi|^2, \quad |a| \le \delta^{-1}.
$$
\end{assumption}

\begin{assumption}
                \label{assumption 1.2}
The function $a$ is of class $L_{\infty}  C^{\alpha/3, \alpha}_{x, v} (\bR^{1+2d}_T)$, and for some $K > 0$,
$$
    [a]_{ L_{\infty} C^{\alpha/3, \alpha}_{x, v} (\bR^{1+2d}_T) } \le K \delta^{-1}.
$$
\end{assumption}

\begin{assumption}
                \label{assumption 1.3}
The functions $b = (b^1 (z), \ldots, b^d (z))$ and $c = c (z)$ are bounded measurable such that
$$
    \|b\|_{ L_{\infty} C^{\alpha/3, \alpha}_{x, v} (\bR^{1+2d}_T) } + \|c\|_{ L_{\infty} C^{\alpha/3, \alpha}_{x, v} (\bR^{1+2d}_T) } \le L
$$
for some $L > 0$.
\end{assumption}

\begin{definition}
For $s \in \bR$, the fractional Laplacian $(-\Delta_x)^s$ is defined as a Fourier multiplier with the symbol $|\xi|^{2s}$.
For  $s \in (0, 1)$ and $u \in L_p (\bR^d)$,  $(-\Delta_x)^s u$ is understood  as a distribution  defined by  duality as follows:
\begin{equation*}
	((-\Delta_x)^{s} u, \phi) = (u, (-\Delta_x)^{s} \phi), \quad \phi \in C^{\infty}_0 (\bR^d).
\end{equation*}
When $s \in (0, 1/2)$,  for any Lipschitz function $u \in \cup_{p \in [1, \infty]} L_p (\bR^d)$,
the  pointwise formula
\begin{equation}
            \label{1.9}
	(-\Delta_x)^s u (x) = N (d, s) \int_{\bR^d} \frac{u (x) - u (x+y)}{|y|^{d+2s}} \, dy
\end{equation}
is valid.
\end{definition}

\begin{theorem}
                \label{theorem 1.5}
Let $\alpha \in (0, 1)$,
and     Assumptions \ref{assumption 1.1} - \ref{assumption 1.3} be satisfied.
Then, the following assertions hold.

$(i)$ For any $u \in \kC^{2, \alpha} (\bR^{1+2d}_T)$, we have
\begin{equation}
                \label{eq1.5.2}
\begin{aligned}
 & [D^2_v u] +  [(-\Delta_x)^{1/3} u] + [\partial_t u - v \cdot D_x u]_{ L_{\infty} C^{\alpha/3, \alpha}_{x, v} (\bR^{1+2d}_T) }
 + \sup_{ (t, v) \in \bR^{1+d}_T} [u (t, \cdot, v)]_{ C^{(2+\alpha)/3} (\bR^d) }\\
 &\le
  N  \delta^{-\theta} \big([P u + b \cdot D_v u + c u]_{ L_{\infty} C^{\alpha/3, \alpha}_{x, v} (\bR^{1+2d}_T)} + \|u\|_{ L_{\infty} (\bR^{1+2d}_T) }\big),
\end{aligned}
\end{equation}
where $[\,\cdot\,] = [\,\cdot\,]_{ \kC^{\alpha} (\bR^{1+2d}_T) }$, $N = N (d, \alpha, K, L)$, and $\theta=\theta(d,\alpha)$.

$(ii)$ There exist numbers
\begin{equation}
                \label{eq1.5.0}
   \lambda_0 = \delta^{-\theta} \widetilde \lambda_0 (d, \alpha, K, L) > 0, \quad
   \theta = \theta (d, \alpha) > 0
\end{equation}
such that for any $u \in \kC^{2, \alpha} (\bR^{1+2d}_T)$ and $\lambda \ge \lambda_0$,
\begin{equation}
                \label{eq1.5.1}
\begin{aligned}
  &  \lambda^{2+\alpha} \|u\|_{ L_{\infty} (\bR^{1+2d}_T) }   + \lambda^{2}  [u] \\
  &\quad + \lambda^{1+\alpha} \|D_v u\|_{ L_{\infty} (\bR^{1+2d}_T) } + \lambda [D_v u] \\
  &\quad+ \lambda^{\alpha}  \|(-\Delta_x)^{1/3} u\|_{ L_{\infty} (\bR^{1+2d}_T) } + [(-\Delta_x)^{1/3} u]\\
  &\quad+  \lambda^{\alpha} \|\partial_t u - v \cdot D_x u\|_{ L_{\infty} (\bR^{1+2d}_T) } + [\partial_t u - v \cdot D_x u]_{ L_{\infty} C^{\alpha/3, \alpha}_{x, v} (\bR^{1+2d}_T) }\\
  &\quad+ \lambda^{\alpha} \| D_v^2 u\|_{ L_{\infty} (\bR^{1+2d}_T) } + [D^2_v u] + \sup_{ (t, v) \in \bR^{1+d}_T} [u (t, \cdot, v)]_{ C^{(2+\alpha)/3} (\bR^d) } \\
  &  \le  N  \delta^{-\theta} \big([P u + b \cdot D_v u + (c  + \lambda^2) u]_{L_{\infty}  C^{\alpha/3, \alpha}_{x, v} (\bR^{1+2d}_T)} \\
 & \quad +  \lambda^{\alpha}  \|P u + b \cdot D_v u + (c  + \lambda^2) u\|_{L_{\infty}  (\bR^{1+2d}_T)}),
\end{aligned}
\end{equation}
where $N =  N (d, \alpha, K)$.

$(iii)$ For any $\lambda \ge \lambda_0$ (see the assertion $(i)$) and $f \in  L_{\infty}  C^{\alpha/3, \alpha}_{x, v} (\bR^{1+2d}_T)$, Eq. \eqref{1.1}.
has a unique solution $u \in \kC^{2, \alpha} (\bR^{1+2d}_T)$.

$(iv)$ For any finite $S < T$ and $f \in   L_{\infty}  C^{\alpha/3, \alpha}_{x, v} ((S, T) \times \bR^{2d})$,
the Cauchy problem
\begin{equation}
    \label{eq1.5.4}
      P u + b \cdot D_v u + (c + \lambda^2) u = f, \quad u (0, \cdot) \equiv 0
\end{equation}
has a unique solution $u \in \kC^{2, \alpha} ((S, T) \times \bR^{2d})$, and, furthermore,
\begin{align*}
   & \|u\|+ \|D_v u\| + \|D^2_v u\| + \|(-\Delta_x)^{1/3} u\| + \|\partial_t u - v \cdot D_x u\|_{L_{\infty}  C^{\alpha/3, \alpha}_{x, v} (\bR^{1+2d}_T)} \\
   &\le N  \delta^{-\theta} \|f\|_{ L_{\infty}  C^{\alpha/3, \alpha}_{x, v}  ((S, T) \times \bR^{2d})},
\end{align*}
where $\|\,\cdot\,\| = \|\,\cdot\,\|_{ \kC^{\alpha} ((S, T) \times \bR^{2d}) }$, and $N = N (d, \alpha, K, L, T-S)$.
\end{theorem}

\begin{remark}
        \label{remark 1.8}
In the case when $a^{i j}= a^{i j} (t)$ and $b \equiv 0, c \equiv 0$, by a scaling argument (see Lemma \ref{lemma 2.0}), we conclude that \eqref{eq1.5.1} holds for any $\lambda > 0$.
By using a compactness argument as in the proof of Theorem  \ref{theorem 1.5}, one can show that the assertion $(iii)$ of the above theorem is also valid for any $\lambda > 0$ in that case.
\end{remark}

\begin{corollary}[Kinetic interpolation inequalities]
            \label{corollary 1.6}
 For any $u \in \kC^{2, \alpha} (\bR^{1+2d}_T)$,
 $D_v u \in \kC^{\alpha} (\bR^{1+2d}_T)$, and, furthermore, for any $\varepsilon > 0$,
\begin{align*}
&  [u]_{ \kC^{\alpha}  (\bR^{1+2d}_T) } \le N  \varepsilon^2 ([\partial_t u - v \cdot D_x u]_{ L_{\infty} C^{\alpha/3, \alpha}_{x, v} (\bR^{1+2d}_T)} + [D^2_v u]_{ L_{\infty} C^{\alpha/3, \alpha}_{x, v} (\bR^{1+2d}_T)})\\
&\quad+ N \varepsilon^{2-\alpha} \|\partial_t u - v \cdot D_x u\|_{ L_{\infty}  (\bR^{1+2d}_T) } + N  \varepsilon^{-\alpha}\|u\|_{ L_{\infty}  (\bR^{1+2d}_T) },\\
 &  [D_v u]_{ \kC^{\alpha}  (\bR^{1+2d}_T) } \le N  \varepsilon ([\partial_t u - v \cdot D_x u]_{ L_{\infty} C^{\alpha/3, \alpha}_{x, v} (\bR^{1+2d}_T)} + [D^2_v u]_{ L_{\infty} C^{\alpha/3, \alpha}_{x, v} (\bR^{1+2d}_T)})\\
&\quad+ N \varepsilon^{1-\alpha} \|\partial_t u - v \cdot D_x u\|_{ L_{\infty}  (\bR^{1+2d}_T) } + N  \varepsilon^{-1-\alpha}\|u\|_{ L_{\infty}  (\bR^{1+2d}_T) },
\end{align*}
where $N = N (d, \alpha)$.
\end{corollary}

It is easy to see that $\kC^{2, \alpha} (G) \subset  \mathbb{C}^{2, \alpha}(G)$ for an open set $G$ of type \eqref{eq1.12}. The following corollary is concerned with the opposite inclusion.
\begin{corollary}['Equivalence' of $\mathbb{C}^{2, \alpha}$ and $\kC^{2, \alpha}$]
\label{corollary 1.8}
 $(i)$ For any $u \in \mathbb{C}^{2, \alpha} (\bR^{1+2d}_T)$, one has $u \in \kC^{2, \alpha} (\bR^{1+2d}_T)$, and, in addition,
  \begin{align*}
       & [D^2_v u]_{\kC^{\alpha} (\bR^{1+2d}_T)}  \\
        &\le N (d, \alpha) ([\partial_t u - v \cdot D_x u]_{ L_{\infty} C^{\alpha/3, \alpha}_{x, v} (\bR^{1+2d}_T)} + [\Delta_v u]_{ L_{\infty} C^{\alpha/3, \alpha}_{x, v} (\bR^{1+2d}_T)}).
        \end{align*}

  $(ii)$ Let $R>0$. If $u \in \mathbb{C}^{2, \alpha} (Q_{R})$, then, for any $r \in (0, R)$,  $u \in \kC^{2, \alpha} (Q_{r})$, and
 $$
       [D^2_v u]_{\kC^{\alpha} (Q_r)}
  \le N (d, \alpha, r, R)  \|u\|_{ \mathbb{C}^{2, \alpha} (Q_{R})}.
  $$
\end{corollary}

\begin{corollary}[Interior Schauder estimate]
            \label{corollary 1.7}
Let $R>0$ and $r\in (0,R)$ be constants.
For any $u \in \kC^{2, \alpha} (Q_{2r})$,
\begin{equation*}
\begin{aligned}
  &\|\partial_t u - v \cdot D_x u\|_{ L_{\infty} C^{\alpha/3, \alpha}_{x, v} (Q_r) } +[u]_{\kC^{\alpha} (Q_r)}+
  [D_v u]_{\kC^{\alpha} (Q_r)} + [D^2_v u]_{\kC^{\alpha} (Q_r)}\\
&\quad  +\sup_{t, v \in (-r^2, 0) \times B_{r}} \|u (t, \cdot, v)\|_{ C^{(2+\alpha)/3} (Q_{r})}\\
  &\le N \delta^{-\theta} (\|P u  + b \cdot D_v u + c u \|_{L_{\infty} C^{\alpha/3, \alpha}_{x, v} (Q_{R})}+ \|u\|_{L_{\infty}  (Q_{R})}),
\end{aligned}
\end{equation*}
where $Q_r  = (-r^2, 0) \times B_{r^3} \times B_r$
and $N = N (d, \alpha, K, L, r, R)$.
\end{corollary}

\subsection{Related works}
                \label{section 1.4}
In this section, we give a brief overview of the literature related to the Schauder estimates for the second-order   nondegenerate parabolic equations and    KFP  equations.

\textit{Classical Schauder estimates.}
This theory asserts that if all the coefficients and the nonhomogeneous term are  H\"older continuous with respect to  all variables, then so are the second-order (spatial) derivatives of the solution. Such estimates can be proved either by using the integral representation of solutions and the bounds of the higher-order derivatives   of the fundamental solution to the heat equation (see, for example, \cite{F_64}) or by  `kernel-free' methods (see \cite{GM_12}, \cite{Kr_96},   \cite{Sch_96}, \cite{S_97}, \cite{T_86}).

\textit{Partial Schauder estimates for elliptic/parabolic equations.} These are results saying that if  the data are H\"older continuous only with respect to some variables, then so are the second-order derivatives (see  \cite{DK_11}, \cite{DK_19}, \cite{Fi_63},  \cite{TW_10}).

\textit{Schauder estimates for parabolic equations with time irregular coefficients.}
In was showed in \cite{B_69} that if for the nondegenerate parabolic equation, the coefficients and the nonhomogeneous term are of class $L_{\infty, t} C^{\alpha}_x$, then the spatial second-order derivatives of the solution belong to the same space. Later, the author of \cite{Kn_81} improved this result by showing that under the same assumptions, the second-order derivatives are H\"older continuous with respect to the space and time variables.  Both papers \cite{B_69} and \cite{Kn_81} are concerned with the interior Schauder estimate. The global estimate (up to the boundary) was established later in \cite{Li_92}. For the related results for  parabolic PDEs with unbounded nonhomogeneous terms or unbounded lower-order  coefficients, we refer the reader to   \cite{Lo_00} and \cite{KrP_10}, respectively. The parabolic systems with  time irregular coefficients are treated  in \cite{DZ_15} (see also \cite{Boc_13}).

\textit{Schauder estimates for the KFP equations with H\"older continuous coefficients.}
A discussion of the H\"older theory and related results for the KFP equation can be found in \cite{AP_20}.
The global (partial) parabolic Schauder estimate (cf. \cite{B_69}) is established in \cite{L_97} under the additional assumptions that the leading coefficients $a^{i j}$ are independent of time and have a limit at infinity (see also \cite{P_09} and the references therein). In the case when the leading coefficients are H\"older continuous in $t, x$, and $v$, the interior Schauder estimate   was proved in   \cite{DFP_06}, \cite{M_97}, and \cite{HS_20}. Later, the authors of \cite{IM_21}  established the global Schauder estimate in  the H\"older space  $C^{2, \alpha}_l (\bR^{1+2d})$, which is similar to $\kC^{2, \alpha} (\bR^{1+2d})$ (see Remark \ref{remark 1.5}). However, due to the nonequivalence  of the kinetic H\"older spaces and the usual H\"older spaces, the classical theory developed  in \cite{IM_21} does not even  yield the \textit{global} estimate in the case when $d = 1$, $a \equiv 1$, $b \equiv 0, c \equiv 0$, and $f = f (x)$ is smooth, say $f (x) = \sin (x)$. In particular, Theorem 3.5 of \cite{IM_21} requires $f \in C^{\alpha}_l (\bR^{ 3 })$ (see Definition 2.2 therein). It is easy to see that for $\alpha \in (0, 1)$, the  $C^{\alpha}_l (\bR^{3})$ seminorm is equivalent to the one defined in \eqref{1.15}, and, therefore, $\sin (x)$ does not belong to $C^{\alpha}_l (\bR^{3})$. We mention that the authors of \cite{IM_21} used a kernel-free approach inspired by Safonov's proof of the classical Schauder estimate (see the exposition in \cite{Kr_96}).

\textit{Schauder estimates for the KFP equation with  irregular coefficients.}
The partial parabolic Schauder estimates similar to that of \cite{B_69} were investigated in \cite{BB_22}, \cite{CHM_21},  and \cite{HW_22}. Their results can be stated in the following general way:  under the assumptions \ref{assumption 1.1} - \ref{assumption 1.3}, the $L_{\infty} C^{\alpha/3, \alpha}_{x, v}$ seminorm of $D^2_v u$ is controlled by the $L_{\infty} C^{\alpha/3, \alpha}_{x, v}$ norms of $a, b, c$, and $u$. To elaborate,
\begin{itemize}
    \item \cite{HW_22} is concerned with the interior estimate, which is applied to the well-posedness problem for the LE with a `rough' initial datum,

    \item in \cite{BB_22} and \cite{CHM_21}, for $T < \infty$, the global results in $L_{\infty} C^{\alpha/3, \alpha}_{x, v}  (\bR^{1+2d}_T)$  and $L_{\infty} C^{\alpha/3, \alpha}_{x, v} ((0, T) \times \bR^{2d})$, respectively,  were established,

    \item  a certain interior Schauder estimate in all variables $t, x, v$ was proved in \cite{BB_22}, and the authors of \cite{HW_22} also commented on the possibility of deriving such an  estimate  from one of their main results (see the paragraph under the formula $(1.5)$ therein).
    \end{itemize}
We also mention the article \cite{PRS_22}, where the interior Schauder estimate for the operator \eqref{1.2}  was derived under the assumption that the leading coefficients satisfy a  Dini type condition.
A few remarks in order.
\begin{itemize}
\item The papers \cite{BB_22}, \cite{CHM_21},  and \cite{PRS_22}  are concerned with the degenerate Kolmogorov  operators that are more general than \eqref{1.2}.

\item  The arguments of the articles \cite{BB_22},  \cite{CHM_21}, \cite{HW_22} (partial Schauder estimates for the KFP equation) use the explicit form of the fundamental solution of $P$.
\end{itemize}

\textit{Schauder estimate for nonlocal kinetic equations.} For the related results, see \cite{HMZ_20} and \cite{IS_21}, and the references therein.

\subsection{Novelties of the present  work.}
                                        \label{section 1.7}
First, we prove the global and the interior $\kC^{2, \alpha}$  estimates  for Eq. \eqref{1.1} with $a, b, c \in L_{\infty} C^{\alpha/3, \alpha}_{x, v} ( \bR^{1+2d}_T)$ with $T \le \infty$ (see Theorem \ref{theorem 1.5}). 
Since $\kC^{\alpha} \subset L_{\infty} C^{\alpha/3, \alpha}_{x, v}$, our result generalizes those of \cite{BB_22}, \cite{CHM_21}, \cite{HW_22} in the case when the equation considered is \eqref{1.1}.
Second, we show that the constant on the right-hand side of the a priori estimate \eqref{eq1.5.2} grows as $\delta^{-\theta}$ as we decrease the lower eigenvalue bound $\delta$, which is useful when nonlinear equations are considered (cf. \cite{DGY_21}, \cite{DGO_20}).   Third, our approach,  inspired by Campanato's method (see, for example, \cite{GM_12}), is kernel-free and is different from the ones in the existing literature on the KFP equations (cf. \cite{BB_22}, \cite{CHM_21}, \cite{HW_22}, \cite{IM_21}, \cite{PRS_22}).

\subsection{Strategy of the proof.}
                        \label{section 1.3}
The main part of the argument is the proof of the a priori estimate \eqref{eq1.5.2} for  a sufficiently regular function $u$ (see Lemma \ref{lemma 4.1}). We remark that  the $\kC^{\alpha}$ estimate of  $(-\Delta_x)^{1/3} u$ is obtained as a by-product of our argument. Nevertheless, due to Lemma \ref{lemma 3.5},  the mean-oscillation estimate  of $(-\Delta_x)^{1/3} u$ (see Proposition \ref{proposition 3.3}) plays an important role in the proof of $\kC^{\alpha}$ estimate of $D^2_v u$.
To prove \eqref{eq1.5.2}, we follow Campanato's approach (see \cite{GM_12} and \cite{DZ_15}), which enables us to reduce the problem to  estimating a `kinetic' Campanato type seminorm of $D^2_v u$ (see  Lemma \ref{lemma 2.1}) adapted to the symmetries of the KFP operator $P$ (see Lemma \ref{lemma 2.0}).

First, we show how our argument works in the case when the coefficients $a^{i j}$ depend only on the temporal variable.
Our goal is to estimate the mean-oscillation of $(-\Delta_x)^{1/3} u$ and $D^2_v u$ over an arbitrary kinetic cylinder  $Q_r (z_0)$, $z_0 \in \overline{\bR^{1+2d}_T}$.
We split $u$ into a `caloric part' $u_c$ and a remainder $u_{\text{rem}}$ such that
\begin{align*}
  &  P u_c (z) = \chi (t) \quad \text{in} \, (t_0 -  (\nu r)^2, t_0) \times \bR^d \times B_{\nu r} (v_0),\\
  & P u_{\text{rem}}(z) = \big(f (z)  - \chi (t)\big) \phi (t, v) \quad \text{in} \, (t_0 -  (2\nu r)^2, t_0)  \times \bR^{2d},
\end{align*}
Here
\begin{itemize}
\item[--] $f = P u$, $\chi (t) = f (t, x_0 - (t-t_0)v_0, v_0)$,
\item[--] $\phi$ is a suitable cutoff function,
\item[--] $\nu \ge 2$ is a number, which  we will choose later.
\end{itemize}
By using the $S_2$ estimate (see Theorem \ref{theorem A.1}), we bound the $L_2$ average of $D^2_v u_{\text{rem}}$ and $(-\Delta_x)^{1/3} u_{\text{rem}}$ over the cylinder $Q_r (z_0)$.
Furthermore,  by the $S_2$ regularity results and the pointwise formula \eqref{1.9} for the fractional Laplacians, we prove the mean-oscillation estimate for $D^2_v u_{c}$ and $(-\Delta_x)^{1/3} u_{c}$. Combining these bounds, we obtain the mean-oscillation inequality for $D^2_v u$ and $(-\Delta_x)^{1/3} u$ (see Proposition \ref{proposition 3.3}).  Taking $\nu \ge 2$ large and using the equivalence of the Campanato and H\"older seminorms (see Lemma \ref{lemma 2.1}), we prove \eqref{eq1.5.2}. We remark that the choice of  the function $\chi$  is dictated by the specific form of the kinetic cylinder $Q_r (z_0)$. In the spatially homogeneous case, one can take $\chi (t) = f (t, v_0)$ (see \cite{DZ_15}).

In the general case,  we  perturb the mean-oscillation estimates in Proposition \ref{proposition 3.3} by using the method of frozen coefficients (see Lemma \ref{lemma 4.1}) and follow the above argument.

\subsection{Additional notation and remarks}

\textbf{Geometric notation.}
                        \label{section 1.5}
\begin{align}
 &  	B_r (x_0) = \{\xi \in \bR^d: |\xi-x_0| < r\},
 \quad  B_r  = B_r (0), \notag\\
&\label{eq1.10}
	Q_{r,  c r} (z_0)
	=  \{z: -r^2<t - t_0<0, |v-v_0| < r,  |x - x_0 + (t - t_0) v_0|^{1/3} < c r\},\\
& \label{eq1.13}
\widetilde Q_{r, c r} (z_0) = \{z: |t - t_0| < r^2, |v-v_0| < r,  |x - x_0 + (t - t_0) v_0|^{1/3} < c  r\}, \\
&\label{eq1.14}
\widehat{Q}_{r} (z_0) = \{z \in \bR^{1+2d}:   \widehat \rho (z, z_0) < r\},\\
 & Q_{r,  c r}  = Q_{r,  c r} (0), \quad
 \widetilde Q_{r, c r}  = \widetilde Q_{r, c r} (0), \quad \widetilde Q_r (z_0) = \widetilde Q_{r, r} (z_0).  \notag
\end{align}

\textbf{Average.} For a function $f$ on $\bR^d$ and a Lebesgue measurable set $A$ of positive finite measure, we denote its average over $A$ as
$$
    (f)_A = \fint_{A} f \, dx = |A|^{-1} \int_A f \, dx.
$$

\textbf{Functional spaces.}
 For an open set $G \subset \bR^d$, we set $C_b (\overline{G})$ to  be the space of all bounded uniformly continuous functions on $\overline{G}$.
Furthermore, for $k \in \{1, 2, \ldots\}$, we denote  by  $C^k_b (\overline{G})$ the space  of all functions in $C_b (\overline{G})$ such that   all the derivatives up to order $k$ extend continuously to $\overline{G}$.
We also set $C^k_0 (\bR^d)$ to be the subspace of all $C^k_b (\bR^d)$ functions vanishing at infinity along with all the derivatives  up to order $k$.

\textit{Kinetic Sobolev spaces.}
For $p \in [1, \infty]$ and an open set $G \subset \bR^{1+2d}$,
\begin{equation}
            \label{eq1.11}
    S_p (G) := \{u \in L_p (G): \partial_t u - v \cdot D_x u, D_v u, D^2_v u \in L_p (G)\}.
\end{equation}

\textit{Local   kinetic Sobolev spaces.}
By $L_{p; \text{loc}} (G)$ we denote the set of all measurable functions $u$ such that for any $\phi \in C^{\infty}_0 (G)$, $u \phi \in L_p (G)$. Furthermore, we define $S_{p; \text{loc}} (G)$ by \eqref{eq1.11} with $L_{p} (G)$ replaced with $L_{p; \text{loc}} (G)$.

\begin{remark}
            \label{remark 1.1}
Here we give a couple of examples of functions belonging to the spaces $\kC^{\alpha} (\bR^{1+2d}_T)$ and $\kC^{2, \alpha} (\bR^{1+2d}_T)$.

As pointed out in Section \ref{section 1.4}, even if $u = u (x, v)$ is smooth in $x$ and $v$, it might not be of class $\kC^{\alpha} (\bR^{1+2d}_T)$. On the other hand, it is easy to prove directly that for $\zeta, \xi \in C^{\infty}_0 (\bR^d)$, one has  $\zeta (x) \xi (v) \in \kC^{\alpha} (\bR^{1+2d}_T)$. This fact also follows from Lemma \ref{lemma 5.1}. Similarly, one can also show that $\zeta (x) \xi (v) \in \kC^{2, \alpha} (\bR^{1+2d}_T)$.

Here is an example of a function of class $\kC^{2, \alpha} (\bR^{1+2d}_T)$ that depends on all variables $t, x, v$.  Let $\psi \in C^3_b (\bR^d)$ and denote
$$
    \phi (z) = e^{-t^2} \psi (x+tv).
$$
We have
$$
    \partial_t \phi  - v \cdot D_x \phi = -2t e^{-t^2} \psi (x+v t),
$$
$$
    D_{v_i v_j} \phi (z) = t^2 e^{-t^2} (D_{v_i v_j} \psi) (x+t v).
$$
Again, either estimating the $\kC^{\alpha}$ seminorm directly or by using Lemma \ref{lemma 5.1}, we conclude that $u, D^2_v u, \partial_t u - v \cdot D_x u \in \kC^{\alpha} (\bR^{1+2d}_T)$.
\end{remark}

\begin{remark}
            \label{remark 1.2}
It follows from the interpolation inequality in the usual H\"older space (see Lemma \ref{lemma B.3}) that
if $u \in \kC^{2, \alpha} (\bR^{1+2d}_T)$, then for any $\varepsilon > 0$, one has
\begin{align*}
&\|D^2_v u\|_{ L_{\infty} (\bR^{1+2d}_T) } \le N \varepsilon^{\alpha} \sup_{t \in (-\infty, T], x \in \bR^d} [D^2_v u (t, x, \cdot)]_{  C^{\alpha}_v (\bR^d)   } +  N  \varepsilon^{-2} \|u\|_{ L_{\infty} (\bR^{1+2d}_T) }, \\
       & \|D_v u\|_{ L_{\infty} (\bR^{1+2d}_T) }\le N \varepsilon^{1+\alpha} \sup_{t \in (-\infty, T], x \in \bR^d} [D^2_v u (t, x, \cdot)]_{  C^{\alpha}_v (\bR^d)   } +  N  \varepsilon^{-1} \|u\|_{ L_{\infty} (\bR^{1+2d}_T) },
  \end{align*}
  and this is why the suprema of $D_v u$ and $D^2_v u$ are not included in the $C^{2, \alpha} (\bR^{1+2d}_T)$ norm.
\end{remark}

\begin{remark}
                    \label{remark 1.3}
The completeness of $\kC^{\alpha} (\bR^{1+2d}_T)$ and $\kC^{2, \alpha}(\bR^{1+2d}_T)$ follows from that of $L_{\infty} (\bR^{1+2d}_T)$ and the Arzela-Ascoli theorem.
\end{remark}

\begin{remark}
            \label{remark 1.4}
It is easy to see that the following  product rule inequality holds:
$$
    [f g]_{X} \le \|f\|_{ L_{\infty} (\bR^{1+2d}_T)} [g]_{X} + \|g\|_{L_{\infty}(\bR^{1+2d}_T)} [f]_{X},
$$
where $X = \kC^{\alpha} (\bR^{1+2d}_T)$ or $L_{\infty}  C^{\alpha/3, \alpha}_{x, v} (\bR^{1+2d}_T)$.
\end{remark}

\subsection{Organization of the paper}
                    \label{section 1.6}
  In Section \ref{section 2}, we prove some auxiliary results including the equivalence of the kinetic  H\"older and Campanato seminorms.
  In Section \ref{section 3}, we establish the mean-oscillation estimates of $(-\Delta_x)^{1/3} u$ and $D^2_v u$ which constitute the crux of the proof of Theorem \ref{theorem 1.5}. We give a proof of the aforementioned theorem in Section \ref{section 4}. Finally, Corollaries \ref{corollary 1.6} - \ref{corollary 1.7}  are proved in Section \ref{section 5}.

\textbf{Acknowledgements.} The authors would like to thank  Weinan Wang for drawing their attention to the article \cite{HW_22}.

\section{Auxiliary result}
                \label{section 2}

\begin{lemma}
			\label{lemma 2.0}
Let $p\in [1,\infty]$  and
$u \in S_{p, \text{loc} } (\bR^{1+2d}_T)$.
For any $z_0 \in \bR^{1+2d}_T$ and any function $h$ on $\bR^{1+2d}_T$, denote
\begin{align}
            \label{eq2.0.0}
&	\widetilde z = (r^2 t + t_0, r^3 x + x_0 -  r^2 t v_0, r v + v_0),\quad
	\widetilde h (z) = h (\widetilde z),\\
            \label{eq2.0.1}
 &   Y = \partial_t  - v \cdot D_x,
\quad    \widetilde P = \partial_t
     - v \cdot D_x
     - a^{ij} (\widetilde z) D_{   v_i v_j }.
\end{align}
Then,
$$
    Y \widetilde u (z) = r^2 Y u (\widetilde z),
\quad		\widetilde P \widetilde u (z) = r^2 (P u) (\widetilde z).
$$
\end{lemma}

We introduce a kinetic Campanato type seminorm
\begin{equation}
            \label{eqB.1.10}
        [u]_{ \mathcal{L}^{2, \alpha}_{\text{kin}} (\bR^{1+2d}_T) } = \sup_{r > 0, z_0 \in \overline{\bR^{1+2d}_T} }  r^{-\alpha}
    (|u - (u)_{Q_r (z_0)}|^2)^{1/2}_{Q_r (z_0)}
\end{equation}
 (cf. Chapter 5 in \cite{GM_12}).

Here is a version of  Campanato's result (cf. Theorem 5.5 in \cite{GM_12}).
\begin{lemma}
            \label{lemma 2.1}
Let $\alpha \in (0, 1]$ and $u \in L_{2, \text{loc}} (\bR^{1+2d}_T)$ be a function
such that
$$
[u]_{ \mathcal{L}^{2, \alpha}_{\text{kin}} (\bR^{1+2d}_T)} < \infty.
$$
Then, one has
\begin{equation}
    \label{eqB.1.0}
    N [u]_{\kC^{\alpha} (\bR^{1+2d}_T)} \le  [u]_{ \mathcal{L}^{2, \alpha}_{\text{kin}} (\bR^{1+2d}_T)}
  \le  N^{-1}  [u]_{\kC^{\alpha} (\bR^{1+2d}_T)},
\end{equation}
where  $N = N (d, \alpha)$.
\end{lemma}

\begin{proof}
The second estimate follows from the  definitions of the seminorms.
The proof of the first bound is split into three steps.

\textit{Step 1: replacing $\rho$ with its symmetrization $\widehat \rho$ (see \eqref{1.3} -\eqref{1.4}).}
We claim that to prove \eqref{eqB.1.0}, it suffices to show that for any $z_1, z_2 \in \bR^{1+2d}_T$,
\begin{equation}
\begin{aligned}
    \label{eqB.1.2}
    &|u (z_1) - u (z_2)| \le \widehat \rho^{\alpha} (z_1, z_2) \sup_{r > 0} r^{-\alpha}
     (|u - (u)_{ \widehat Q_r (z_0) \cap \bR^{1+2d}_T}|^2)^{1/2}_{\widehat Q_r (z_0) \cap \bR^{1+2d}_T}.
\end{aligned}
\end{equation}
Assuming \eqref{eqB.1.2},  by Lemma \ref{lemma B.1} $(ii)$, we only need to demonstrate that
\begin{equation}
\begin{aligned}
    \label{eqB.1.4}
     \sup_{r > 0} r^{-\alpha}
     (|u - (u)_{ \widehat Q_r (z_0) \cap \bR^{1+2d}_T}|^2)^{1/2}_{\widehat Q_r (z_0) \cap \bR^{1+2d}_T}
    \le N (d, \alpha) [u]_{ \mathcal{L}^{2, \alpha}_{\text{kin}} (\bR^{1+2d}_T)}.
\end{aligned}
\end{equation}
 Indeed, by Lemma \ref{lemma B.1}  $(iv)$,
\begin{equation}
\label{eqB.1.7}
\begin{aligned}
&
   (|u - (u)_{ \widehat Q_r (z_0) \cap \bR^{1+2d}_T}|^2)^{1/2}_{\widehat Q_r (z_0) \cap \bR^{1+2d}_T}\\
 &
   \le N \frac{|\widetilde Q_r (z_0) \cap \bR^{1+2d}_T|^2}{|\widehat Q_r (z_0) \cap \bR^{1+2d}_T|^2} (|u - (u)_{ \widetilde Q_r (z_0) \cap \bR^{1+2d}_T}|^2)^{1/2}_{\widetilde Q_r (z_0) \cap \bR^{1+2d}_T}.
\end{aligned}
\end{equation}
By the doubling property (see Lemma \ref{lemma B.1} $(v)$) and Lemma \ref{lemma B.1} $(iv)$,
$$
     \frac{|\widetilde Q_r (z_0) \cap \bR^{1+2d}_T|}{|\widehat Q_r (z_0) \cap \bR^{1+2d}_T|}
     \le  \frac{|\widehat Q_{3r} (z_0) \cap \bR^{1+2d}_T|}{|\widehat Q_r (z_0) \cap \bR^{1+2d}_T|}  \frac{|\widetilde Q_{r} (z_0) \cap \bR^{1+2d}_T|}{|\widehat Q_{3r} (z_0) \cap \bR^{1+2d}_T|} \le N (d).
$$
Hence, we may replace the left-hand side of \eqref{eqB.1.4} with
$$
  (|u - (u)_{ \widetilde Q_r (z_0) \cap \bR^{1+2d}_T}|^2)^{1/2}_{\widetilde Q_r (z_0) \cap \bR^{1+2d}_T}.
$$

Next, we will consider the case $T < \infty$ and assume that $T = 0$, for the sake of simplicity.
Note that if $t_0 < - r^2$, one has
$$
	 \tQ_{r} (z_0)  \subset Q_{2r} (t_0 + r^2, x_0 -  r^2 v_0, v_0) \subset \bR^{1+2d}_0.
$$
If $t_0 \ge -r^2$, then,
$$
	\tQ_{r} (z_0) \cap \bR^{1+2d}_0 \subset \overline{ Q_{2 r}} (0, x_0 + t_0 v_0, v_0).
$$
Thus,
$$
   \sup_{r > 0, z_0 \in \bR^{1+2d}_T} r^{-\alpha}  (|u - (u)_{ \widetilde Q_r (z_0) \cap \bR^{1+2d}_T}|^2)^{1/2}_{\widetilde Q_r (z_0) \cap \bR^{1+2d}_T} \le N (d, \alpha) [u]_{ \mathcal{L}^{2, \alpha}_{\text{kin}} (\bR^{1+2d}_T) },
$$
so that \eqref{eqB.1.4} holds.

\textit{Step 2: estimate of the deviation of $u$ from its  average.}
In the remaining steps, we follow the argument of Theorem 5.5 in \cite{GM_12} closely.
Here we prove that for a.e. $z_0 \in \bR^{1+2d}_T$, and $r > 0$,
\begin{equation}
            \label{eqB.1.1}
|u (z_0) - (u)_{\widehat Q_r (z_0) \cap \bR^{1+2d}_T }|  \le r^{\alpha} [u]_{ \mathcal{L}^{2, \alpha}_{\text{kin}} (\bR^{1+2d}_T) }.
\end{equation}

First, let $r_n = 2^{-n} r$ and denote $\mathcal{Q}_n (z_0) = \widehat Q_{r_n} (z_0) \cap \bR^{1+2d}_T$. We claim that
\begin{equation}
            \label{eqB.1.5}
    |(u)_{ \mathcal{Q}_{n} (z_0) } - (u)_{ \mathcal{Q}_{n+1} (z_0) }|
    \le N (d) \, r_{n+1}^{\alpha} [u]_{ \mathcal{L}^{2, \alpha}_{\text{kin}} (\bR^{1+2d}_T) }.
\end{equation}
To prove this, we note that for any Lebesgue measurable sets of finite measure $A \subset A'$,
\begin{equation}
            \label{eqB.1.3}
    |(f)_{A'} - (f)_A| \le \frac{|A'|}{|A|} (|f - (f)_{A'}|)_{A'} \le \frac{|A'|}{|A|} (|f - (f)_{A'}|^2)^{1/2}_{A'}.
\end{equation}
This combined with the doubling property (see Lemma \ref{lemma B.1} $(v)$) yields \eqref{eqB.1.5}.
Then, by  using telescoping series and \eqref{eqB.1.5}, we obtain
\begin{equation}
    \label{eqB.1.6}
\begin{aligned}
  &  |(u)_{ \widehat {Q}_{r} (z_0) \cap \bR^{1+2d}_T} - (u)_{ \mathcal{Q}_{n+1} (z_0) }| \le \sum_{j=0}^n |(u)_{ \mathcal{Q}_{j} (z_0)} - (u)_{ \mathcal{Q}_{j+1} (z_0) }|\\
   & \le N (d, \alpha)r^{\alpha}  [u]_{ \mathcal{L}^{2, \alpha}_{\text{kin}} (\bR^{1+2d}_T) } \sum_{j=0}^n 2^{-\alpha (j+1)} \le N (d, \alpha)\, r^{\alpha} [u]_{ \mathcal{L}^{2, \alpha}_{\text{kin}} (\bR^{1+2d}_T)}.
\end{aligned}
\end{equation}
Furthermore, by the Lebesgue differentiation theorem in spaces of homogeneous type (see Lemma 7 in \cite{C_76}) and  Lemma \ref{lemma B.1} $(v)$,
$$
    \lim_{R \to 0} (u)_{ \widehat {Q}_{R} (z_0) \cap \bR^{1+2d}_T } = u (z_0) \quad \text{for a.e.} \, \, z_0 \in \bR^{1+2d}_T.
$$
Then, passing to the limit in \eqref{eqB.1.6} as $n \to \infty$, we prove \eqref{eqB.1.1}.

\textit{Step 3: proof of \eqref{eqB.1.2}.}
We fix any two points $z_1, z_2 \in \bR^{1+2d}_T$ satisfying \eqref{eqB.1.1}  and  denote $r = \widehat \rho (z_1, z_2)$. In view of Lemma \ref{lemma B.1}, we have $\widehat {Q}_{r} (z_1)\subset \widehat {Q}_{4r} (z_2)$.
Then, by the triangle inequality,
\begin{equation}
\begin{aligned}
            \label{eqB.1}
        |u (z_1) - u (z_2)| &\le |u (z_1) - (u)_{ \widehat {Q}_{r} (z_1) \cap \bR^{1+2d}_T}|+ |u (z_2) - (u)_{\widehat {Q}_{ 4 r} (z_2) \cap \bR^{1+2d}_T}|\\
       &\quad  + |(u)_{\widehat {Q}_{ 4 r} (z_2) \cap \bR^{1+2d}_T} - (u)_{\widehat {Q}_{r} (z_1) \cap \bR^{1+2d}_T}| =: J_1+J_2+J_3.
        \end{aligned}
\end{equation}
By \eqref{eqB.1.1}, we have
\begin{equation}
            \label{eqB.2}
    J_1+J_2 \le N (d, \alpha)\,  r^{\alpha} [u]_{ \mathcal{L}^{2, \alpha}_{\text{kin}} (\bR^{1+2d}_T) }.
\end{equation}

Next, to estimate $J_3$, we use an argument similar to that of \eqref{eqB.1.7}. By Lemma \ref{lemma B.1}, \eqref{eqB.1.3}, and the doubling property (Lemma \ref{lemma B.1} $(v)$), we obtain
\begin{equation}
        \label{eqB.3}
\begin{aligned}
    J_3 &\le N (d)   \frac{|\widehat Q_{  4 r} (z_{ 2 }) \cap \bR^{1+2d}_T|^2}{|\widehat Q_r (z_{1}) \cap \bR^{1+2d}_T|^2} (|u - (u)_{\widehat Q_{ 4 r} (z_{ 2  })  \cap \bR^{1+2d}_T }|^2)^{1/2}_{\widehat Q_{  4 r} (z_{2})   \cap \bR^{1+2d}_T } \\
  & \le  N (d)  \,    r^{\alpha} [u]_{ \mathcal{L}^{2, \alpha}_{\text{kin}} (\bR^{1+2d}_T) }.
\end{aligned}
\end{equation}
Combining \eqref{eqB.1} - \eqref{eqB.3}, we prove \eqref{eqB.1.2} for a.e. $z_1, z_2 \in \bR^{1+2d}_T$. By continuity argument, \eqref{eqB.1.2} holds for all $z_1, z_2$.
\end{proof}

\section{Estimate for the model equation}
                               \label{section 3}
In this section, we assume that the coefficients $a^{i j}$ are independent of $x, v$ and satisfy Assumption  \ref{assumption 1.1}. We denote
\begin{equation}
            \label{eq3.0}
    P_0  = \partial_t - v \cdot D_x - a^{i j} (t) D_{v_i v_j}.
\end{equation}
Our goal  is to prove a mean-oscillation estimate for $(-\Delta_x)^{1/3} u$ and $D^2_v u$ (see Proposition \ref{proposition 3.3}).
As explained in Section \ref{section 1.3}, we split $u$ into a `caloric part' $u_c$ and a remainder $u_{\text{rem}}$.
The mean-square estimate of $u_{\text{rem}}$  is proved via Lemma \ref{lemma 3.2}.
To estimate the mean-square oscillation of $u_c$, we need to modify  the argument of Section 5 in \cite{DY_21a}.

\begin{proposition}
            \label{proposition 3.3}
Let $\nu \geq 2, \alpha \in (0, 1),  r > 0$ be   numbers, $\chi = \chi (t) \in L_{2, \text{loc}} (\bR_T)$, and $u \in S_2 (\bR^{1+2d}_T)$ (see \eqref{eq1.11}).
Then, there exists  $\theta = \theta (d) > 0$ and $N = N (d) > 0$ such for any  $z_0 \in \overline{\bR^{1+2d}_T}$,
  \begin{align}
    \label{eq3.3.0}
    I_1 &:=
    \bigg(|(-\Delta_x)^{1/3} u -  ((-\Delta_x)^{1/3} u)_{ Q_r (z_0) }|^2\bigg)^{1/2}_{Q_r (z_0)}\\
&\leq N  \nu^{-1} \delta^{-\theta}
    (|(-\Delta_x)^{1/3} u - ((-\Delta_x)^{1/3} u)_{Q_{\nu r} (z_0)}|^2)^{1/2}_{ Q_{\nu r } (z_0)}\notag\\
    &\quad + N  \nu^{1+2d}  \delta^{-\theta}  \sum_{k=0}^\infty 2^{-2k} \big(|P_0 u - \chi|^2\big)^{1/2}_{Q_{ 2\nu r, 2^{k+1}/\delta^2 (2\nu r)} (z_0)},\notag\\
       	I_2 &:=
 	\bigg(|D_v^2 u -  (D_v^2 u)_{ Q_r (z_0) }|^2\bigg)^{1/2}_{ Q_r (z_0) } \notag\\
&
  \leq
N   \nu^{-1} \delta^{-\theta}
    (|D_v^2 u - (D_v^2 u)_{Q_{\nu r} (z_0)}|^2)^{1/2}_{ Q_{\nu r} (z_0)  }\notag\\
&\quad    + N \nu^{-1} \delta^{-\theta}
    \sum_{k = 0}^{\infty}
	 2^{-k}
	 (|(-\Delta_x)^{1/3} u - ((-\Delta_x)^{1/3} u)_{ Q_{\nu r, 2^k \nu r} (z_0)}|^2)^{1/2}_{ Q_{\nu r, 2^k \nu r} (z_0)}\notag\\
	 &\quad  + N  \nu^{1+2d}  \delta^{-\theta}  \sum_{k=0}^\infty 2^{-k} \big(|P_0 u - \chi|^2\big)^{1/2}_{Q_{ 2\nu r, 2^{k+1}/\delta^2 (2\nu r)} (z_0)}\notag.
    \end{align}
\end{proposition}

\begin{definition}
                \label{definition 3.0}
For  $-\infty \le T_1 < T_2\le \infty$, we write $u \in \llocxv ((T_1, T_2) \times \bR^{2d})$ if for any $\zeta = \zeta (x, v) \in C^{\infty}_0 (\bR^{2d})$, we have $u \zeta \in L_2 ((T_1, T_2) \times \bR^{2d})$. We define
$\slocxv ((T_1, T_2) \times \bR^{2d})$ in the same way as we defined $S_2 (G)$ (see \eqref{eq1.11}).
\end{definition}

\begin{lemma}[cf. Lemma 5.2 in \cite{DY_21a}]
         \label{lemma 3.2}
Let
\begin{itemize}
\item[--] $R\ge 1$
be a number,\\
\item[--]  $u\in \slocxv ((-1,0)\times \bR^{2d})$ be a function  such that  $u 1_{t < -1} \equiv 0$, and
\begin{equation}
                \label{eq3.2.0}
  \sum_{k=0}^\infty  2^{-2k -(3d/2)k}
\||u| + |D_v u|\|_{L_2(Q_{1,2^{k+1}R/\delta^2})} < \infty,
\end{equation}
\item[--] $f \in \llocxv ((-1,0)\times \bR^{2d})$ be a function vanishing outside $(-1,0)\times \bR^d\times B_1$ and
$(-\Delta_x)^{1/3} u \in \llocxv ((-1,0)\times \bR^{2d})$,\\
\item[--] $u$  satisfy $P_0 u  = f$ in $(-1,0)\times \bR^{2d}$.
\end{itemize}
Then, one has
\begin{equation}
\begin{aligned}
                \label{eq3.2.1}
&\||u|+|D_v u|+|D_v^2 u|\|_{L_2((-1,0)\times B_{R^3}\times B_R)}\\
&\le N (d)  \delta^{-1} \sum_{k=0}^\infty  2^{-k(k-1)/4}R^{-k}
\|f\|_{L_2(Q_{1,2^{k+1}R/\delta^2})},
\end{aligned}
\end{equation}
and, furthermore, there exists $\theta = \theta (d) > 0$ such that
\begin{equation}
\begin{aligned}
   \label{eq3.2.2}
&
(|(-\Delta_x)^{1/3} u|^2)^{1/2}_{Q_{1,R}}
\le N (d) \delta^{-\theta} \sum_{k=0}^\infty  2^{-2k}(f^2)^{1/2}_{Q_{1,2^{k}R/\delta^2}}.
\end{aligned}
\end{equation}
\end{lemma}

\begin{proof}
We may assume that the right-hand side of \eqref{eq3.2.2} is finite.
Let  $\phi_n, n \ge 1,$ be a sequence of $C^{\infty}_0 (\bR^{2d})$ functions satisfying $\phi_n = 1$ in $\widetilde Q_n$ and the bounds
 \begin{equation}
                \label{eq3.2.3}
        |\phi_n| \le N, \quad |D_v \phi_n| \le N/n, \quad
        |\partial_t  \phi_n| \le N/n^2, \quad |D_x \phi_n| \le N/n^3
 \end{equation}
with $N$ independent of $n$.

Note that $u_n: = u \phi_n \in S_2 ((-1, 0) \times \bR^{2d})$ satisfies the identities
$$
    P_0 u_n = f \phi_n + u P_0 \phi_n - 2 (a D_v u) \cdot D_v \phi_n =: f_n, \quad u_n 1_{t < -1} \equiv 0.
$$
Then, by Lemma 5.2 in \cite{DY_21a}, one has
\begin{align}
            \label{eq3.2.4}
&\||u_n|+|D_v u_n|+|D_v^2 u_n|\|_{L_2((-1,0)\times B_{R^3}\times B_R)}\\
&\le N (d)  \delta^{-1} \sum_{k=0}^\infty  2^{-k(k-1)/4}R^{-k}
\|f_n\|_{L_2(Q_{1,2^{k+1}R/\delta^2})}\notag
\end{align}
and
\begin{equation}
  \label{eq3.2.5}
(|(-\Delta_x)^{1/3} u_n|^2)^{1/2}_{Q_{1,R}}
\le N (d) \delta^{-\theta} \sum_{k=0}^\infty  2^{-2k}(f_n^2)^{1/2}_{Q_{1,2^{k}R/\delta^2}}.
\end{equation}

 By \eqref{eq3.2.3}, for any $r > 0$,
\begin{equation}
\begin{aligned}
                \label{eq3.2.6}
    \|f_n\|_{L_2(Q_{1, r})}& \le   \|f\|_{L_2(Q_{1, r})} + N (d, \delta) n^{-1} \||u|+|D_v u|\|_{L_2(Q_{1, r})}. 
    \end{aligned}
\end{equation}
Then, by using this and  \eqref{eq3.2.0}, and passing to the limit as $n \to \infty$  in \eqref{eq3.2.4}, we prove \eqref{eq3.2.1}.

Next, we prove the bound for $(-\Delta_x)^{1/3} u$.
For any   smooth cutoff function $\xi$ supported in $Q_{1, R}$, we have
\begin{align*}
&    \bigg| \int u \big((-\Delta_x)^{1/3} \xi \big) \, dz \bigg|
  =\lim_{n \to \infty}\bigg| \int u_n \big((-\Delta_x)^{1/3} \xi \big) \, dz \bigg|\\
&    \le \nlimsup_{n \to \infty} \|(-\Delta_x)^{1/3} u_n \|_{ L_2 (Q_{1, R})  } \|\xi\|_{ L_2 (Q_{1, R})  }.
\end{align*}
Finally,  due to the last inequality and a duality argument, the left-hand side of \eqref{eq3.2.2} is bounded by the limit  supremum of the right-hand side of \eqref{eq3.2.5} as $n \to \infty$.  Now \eqref{eq3.2.2} follows from the above, \eqref{eq3.2.6}, and \eqref{eq3.2.0}.
\end{proof}

The following `nonlocal' lemma is similar to  Lemma 5.5 of \cite{DY_21a} and Lemma 3.8 in \cite{DY_21b}. In the present authors' opinion, such `nonlocal' lemmas are the  technical novelties of the papers \cite{DY_21a} - \cite{DY_21b}, and the current article.
 \begin{lemma}
			\label{lemma 3.5}
   Let
  $
	u \in S_2((-4, 0)\times \bR^{2d})
  $
  be a function satisfying
$P_0 u = 0$  a.e.  in
  $(-1, 0) \times \bR^d \times B_1$.
   Then, the following assertions hold.

   $(i)$ We have $(-\Delta_x)^{1/3} u \in S_{2, \text{loc}} ((-1, 0) \times \bR^d \times B_1)$,
   and
   $$
    P_0 (-\Delta_x)^{1/3} u = 0 \quad \text{a.e. in} \, \,   (-1, 0) \times \bR^d \times B_1.
   $$

   $(ii)$ For any $r \in (0, 1)$,
   \begin{equation}
				\label{eq3.5.0}
	\begin{aligned}
     &	\| D_x u \|_{ L_2 (Q_r)} \\
   &\leq
	 N (d, r) \delta^{-4}  \sum_{k = 0}^{\infty}
	  2^{-k}	  (|(-\Delta_x)^{1/3} u - ((-\Delta_x)^{1/3} u)_{  Q_{1, 2^{k +2} }  }|^2)_{ Q_{1, 2^{k+2}}}^{ 1/2   },
	 \end{aligned}
   \end{equation}
   where $Q_{1, 2^k}$ is defined in \eqref{eq1.10}.
   \end{lemma}

   \begin{proof}
   First, multiplying $u$ by a suitable cutoff function $\phi = \phi (t)$ and using Corollary  \ref{corollary A.4},  we conclude that $(-\Delta_x)^{1/3} u \in L_2 ((-1, 0) \times \bR^{2d})$,
   and hence, the series on the right-hand side of \eqref{eq3.5.0} converges.

  $(i)$    Let $u_{\varepsilon}$ be the mollification of $u$ in the $x$ variable with the standard mollifier and note that $\partial_t u_{\varepsilon} \in \llocxv ((-4, 0) \times \bR^{2d})$.
  Furthermore, let $\zeta$ be either $u_{\varepsilon}$ or $\partial_t  u_{\varepsilon}$, or $D^2_v u_{\varepsilon}$.
   Then, by the  formula \eqref{1.9}, for a.e. $t, v \in (-1, 0) \times B_1$,
   \begin{itemize}
   \item[--]
   $
         \zeta (t, \cdot, v) \in C^k_b (\bR^d), \quad k  \in \{1, 2 \ldots\},
   $
   \item[--]  $(-\Delta_x)^{1/3} \zeta$ is a well defined function given by \eqref{1.9} with $u$ replaced with $\zeta$,
   \item[--]
   $
         (-\Delta_x)^{1/3} A u_\varepsilon (t, \cdot, v) \equiv A (-\Delta_x)^{1/3} u_\varepsilon (t, \cdot, v), \quad A = \partial_t, D^2_v.
   $
   \end{itemize}
   By  the above facts, we conclude
   \begin{equation}
            \label{eq2.1.1}
        P_0 (-\Delta_x)^{1/3} u_{\varepsilon} = 0 \quad \text{a.e. in} \, \, (-1, 0) \times \bR^d \times B_1.
   \end{equation}
   Consequently, by the interior $S_2$ estimate (see Lemma \ref{lemma A.3}), for any $0 < r < 1$,
    $$
        \||(\partial_t  - v \cdot D_x) (-\Delta_x)^{1/3} u_{\varepsilon}| + |D^2_v  (-\Delta_x)^{1/3} u_{\varepsilon}|\|_{L_2 (Q_r) } \le N  \|(-\Delta_x)^{1/3} u\|_{L_2 (Q_1) },
    $$
    where $N = N (d, \delta, r)$.
   Passing to the limit as $\varepsilon \to 0$ in the above inequality and in \eqref{eq2.1.1}, we prove the assertion $(i)$.

$(ii)$
 We inspect the argument of Lemma 5.5 in \cite{DY_21a}. In the sequel, $N = N (d, r)$.
Let  $\eta \in C^{\infty}_0 (\widetilde Q_{(r+1)/2})$ be a function such that $\eta = 1$ in $Q_r$ and denote
$$
    g = (-\Delta_x)^{1/3} u_{\varepsilon} - ((-\Delta_x)^{1/3} u_{\varepsilon})_{Q_{1, 4}}.
$$
We decompose $\eta^2 D_x u$ in the following way:
\begin{align*}
    	\eta^2 D_x u_{\varepsilon}
    = \eta  (\mathcal{L} g + \text{Comm}),
\end{align*}
where
\begin{align*}
	\mathcal{L} g = \cR_x  (-\Delta_x)^{1/6} (g \eta),\quad
\text{Comm} =  \eta D_x u_{\varepsilon} - \cR_x  (-\Delta_x)^{1/6} (g \eta),
\end{align*}
and  $\cR_x = D_x (-\Delta_x)^{-1/2}$ is the Riesz transform.

\textit{Estimate of $\mathcal{L}  g$.} By \eqref{eq2.1.1},
$$
    P_0 (g \eta) = g P_0 \eta - 2 (a D_v \eta) \cdot D_v g \quad \text{in} \, \, (-1, 0) \times \bR^d \times B_1.
$$
Then, by Theorem \ref{theorem A.1} and the fact that $|a| \le \delta^{-1}$, we have
\begin{align*}
   & \|(-\Delta_x)^{1/3} (g \eta)\|_{ L_2 (\bR^{1+2d}_0)} \le N \delta^{-1} \||g P_0 \eta| + |(a D_v \eta) \cdot D_v g| \|_{ L_2 (\bR^{1+2d}_0)} \\
    & \le N \delta^{-2} \||g|+|D_v g|\|_{ L_2 (Q_{(r + 1)/2})  }.
\end{align*}
Furthermore, by \eqref{eq2.1.1} and the interior $S_2$ estimate in Lemma  \ref{lemma A.3}, the last term is bounded by
$$
    N \delta^{-4} \|g\|_{ L_2 (Q_{1})  }.
$$
Finally, due to the $L_p$-boundedness of the Riesz transform and the H\"ormander-Mikhlin inequality, we have
\begin{equation}
\begin{aligned}
            \label{eq3.5.1}
    \|\mathcal{L} g\|_{ L_2 (Q_{r})  } & \le N (d) \||(-\Delta_x)^{1/3} (\eta g)| + |\eta g|\|_{ L_2 (\bR^{1+2d}_0)}\\
    & \le  N \delta^{-4} \|g\|_{ L_2 (Q_{1})  }.
\end{aligned}
\end{equation}

\textit{Estimate of $\text{Comm}$.} We denote
$
    \mathcal{A} = D_x (-\Delta_x)^{-1/3}.
$
Since $u_{\varepsilon} \in C^2_0 (\bR^d)$ (see Section \ref{section 1.5}) for a.e. $t, v \in (-1, 0) \times B_1$ and $x \in \bR^d$, by Lemma \ref{lemma B.5} $(ii)$,
$$
    D_x g (z) \equiv \mathcal{A} (-\Delta_x)^{1/3} g (z).
$$
Hence, we have
$$
    \text{Comm} =   \eta (\mathcal{A} g) - \mathcal{A} (\eta g).
$$
By  the explicit representation of $\mathcal{A}$ (see Lemma \ref{lemma B.5} $(i)$) and the oddness of the kernel $y |y|^{-d- 4/3}$,
  for any $z  \in Q_r$,  we have
\begin{align*}
      &  \text{Comm} (z) = \int \big(\eta (t, x, v) - \eta (t, x-y, v)\big)
    g (t, x - y, v) \frac{y}{ |y|^{d+4/3}} \, dy = J_1 + J_2\\
    &: =\int_{ |y| < 8  } \big(\eta (t, x, v) - \eta (t, x-y, v)\big)
    g (t, x - y, v) \frac{y}{ |y|^{d+4/3}} \, dy+ \eta (t, x, v)\\
    &\quad \cdot \sum_{k=2}^{\infty} \int_{ 2^{3(k-1)} < |y| < 2^{3k} }
    \big((-\Delta_x)^{1/3} u_{\varepsilon} (t, x - y, v) - ((-\Delta_x)^{1/3} u_{\varepsilon})_{ Q_{1, 2^k}    } \big) \frac{y}{ |y|^{d+4/3}} \, dy.
\end{align*}
By the Minkowski inequality,
\begin{equation}
\begin{aligned}
            \label{eq3.5.2}
    \|J_1\|_{ L_2 (Q_r)  } \le   N (d, r) \|g\|_{ L_2 (Q_{1, 4 }) }.
\end{aligned}
\end{equation}
By the Cauchy-Schwartz inequality, for any $z \in Q_r$,
\begin{align*}
    |J_2 (z)| &\le N (d) \sum_{k=2}^{\infty} 2^{-k} \bigg(\fint_{ 2^{3(k-1)} < |y| < 2^{3k} }
    \big|(-\Delta_x)^{1/3}u_{\varepsilon} (t, x - y, v) \\
    &\qquad\qquad- ((-\Delta_x)^{1/3} u_{\varepsilon})_{ Q_{1, 2^k}    }\big|^2 \, dy\bigg)^{1/2}.
\end{align*}
Then, by using Minkowski inequality again, we get
\begin{equation}
\begin{aligned}
            \label{eq3.5.3}
    \|J_2\|_{ L_2 (Q_r)  } \le N (d) \sum_{k =  2 }^{\infty} 2^{-k}
    \big(|(-\Delta_x)^{1/3}u_{\varepsilon}
    - ((-\Delta_x)^{1/3} u_{\varepsilon})_{ Q_{1, 2^k}    }|^2\big)^{1/2}_{Q_{1, 2^k}}.
\end{aligned}
\end{equation}

 Finally, combining \eqref{eq3.5.1} - \eqref{eq3.5.3}, we obtain \eqref{eq3.5.0} with $u$ replaced with $u_{\varepsilon}$. Passing to the limit as $\varepsilon \to 0$, we prove \eqref{eq3.5.0}.
   \end{proof}

\begin{lemma}[Lemma 5.6 $(i)$  in \cite{DY_21a}]
            \label{lemma 3.7}
Let  $u\in S_{2,\text{loc}}((-1, 0) \times \bR^{2d})$
be a function such that
$P_0 u   = 0$
in $(-1, 0) \times \bR^d \times B_1$.
Then  for any $m, l \ge 0$ and $j=0,1$, there exists   $\theta = \theta (d, j, l, m) > 0$ such that for any $R \in (1/2, 1]$,
$$ \|\partial_t^j D_x^l D_v^{m}  u\|_{ L_{\infty} (Q_{1/2})  }
\leq	
N (d, j, l, m,  R)  \delta^{-\theta} \|u\|_{ L_2 (Q_R) }.
$$
\end{lemma}

\begin{lemma}
            \label{lemma 3.6}
Let  $u\in S_{2,\text{loc}}((-4, 0) \times \bR^{2d})$
be a function such that
$P_0 u (z)  = \chi$
in $(-1, 0) \times \bR^d \times B_1$, where $\chi = \chi (t)$.
Then, for any $l, m \ge 0$ and $j=0,1$ such that $j+l+m \ge 1$, there exists $\theta = \theta (d, j, l, m) > 0$ such that
\begin{equation}
    \begin{aligned}
           \label{eq3.6.0}
   \, & \|\partial_t^j D^l_x D_v^{m+2}   u\|_{ L_{\infty} (Q_{1/2})  }
    \\
  &  \leq N (d, j,  l, m)   \delta^{-\theta}  \big( \|D_v^2 u   - (D_v^2 u)_{Q_1} \|_{L_2 (Q_1)} + \|D_x u\|_{L_2 (Q_1)}\big).\\
\end{aligned}
\end{equation}
\end{lemma}

\begin{proof}
\textbf{Step 1: $L_2$ estimate of derivatives.}
Here we  will show that for  $j \in \{0, 1\}$ and $l+m \ge 1$, and $1/2 \le r <  R \le 1$,
\begin{equation}
            \label{eq2.5.8}
   \|\partial_t^j D^l_x D_v^{m}   u\|_{L_2 (Q_{r}) } \le N  \delta^{-\theta} (\|D_v u\|_{L_2 (Q_{R}) } +  \|D_x u\|_{L_2 (Q_{R}) }).
\end{equation}
To do that, we  follow the argument of Lemma 5.6 in \cite{DY_21a}. By mollifying $u$ in  the $x$ variable, we may assume that $u$ is smooth as a function of $x$.

\textit{Case 1: $j = 0 =l, m \ge 1$.} We will show  that  for any $m \ge 1$,
\begin{equation}
                \label{eq2.5.1}
    \|D^{m}_v  u\|_{L_2 (Q_{r}) }  \le
     N  \delta^{-\theta}  (\|D_x u\|_{L_2 (Q_{R}) }  +  \|D_v u\|_{L_2 (Q_{R}) }),
\end{equation}
where $N = N (d,  r, R)$. We prove this inequality by induction.
Obviously, the  estimate  holds for $m  = 1$.
Furthermore, for any multi-index $\alpha$ of order $m \ge 1$, one has
\begin{equation}
                \label{eq2.5.6}
     P_0 (D_{v}^{\alpha}  u)  =  \sum_{  \widetilde \alpha:  \, \widetilde \alpha < \alpha, |\widetilde \alpha| = m - 1 } c_{\widetilde \alpha}
	  D^{\widetilde \alpha}_v D^{\alpha - \widetilde \alpha}_{x} u.
\end{equation}
By the interior $S_2$ estimate in Lemma \ref{lemma A.3}, for $r < r_1 < R$,
\begin{equation}
                \label{eq2.5.2}
    \|D^{m+1}_v u\|_{L_2 (Q_{r}) } \le
       N  \delta^{-2}  (\|D^{m}_v  u\|_{L_2 (Q_{r_1}) }
   +  \|D^{m-1}_v D_x  u\|_{L_2 (Q_{r_1}) }).
\end{equation}
Note that the first term on the right-hand side of \eqref{eq2.5.2} is bounded by the right-hand side in the equality \eqref{eq2.5.1} by the induction hypothesis.
To handle the second term, note that for any nonempty multi-index $\beta$,
\begin{equation}
            \label{eq2.5.3}
    P_0 (D_x^{\beta} u) = 0 \quad  \text{in} \, \, (-1, 0) \times \bR^d \times B_1.
\end{equation}
Then, by Lemma \ref{lemma 3.7},  for some $r_1 < r_2 < 1$,
\begin{equation}
            \label{eq2.5.6.1}
    \|D^{m-1}_v D_x  u\|_{L_2 (Q_{r_1}) } \le N  \delta^{-\theta} \|D_x u\|_{L_2 (Q_{r_2}) }.
\end{equation}
Thus, the inequality \eqref{eq2.5.1} is valid. To make this  argument rigorous, one  can use the method of finite difference quotients.

\textit{Case $j = 0, l \ge 1, m \ge 0$. } Arguing as in \eqref{eq2.5.6.1} and using \eqref{eq2.5.3} and Lemma \ref{lemma 3.7}, we get
\begin{equation}
                \label{eq2.5.4}
        \|D^{m}_v D_x^l u\|_{  L_2 (Q_r) } \le
   N  \delta^{-\theta} \| D_x  u\|_{  L_2 (Q_R) }.
\end{equation}

\textit{Case 3: $j = 1$, $l + m \ge  1$.}
Note that the function $U = D_x^{\beta} D_{v}^{\alpha} u$, where $|\alpha|=m$ and $|\beta|=l$, satisfies the identity (see \eqref{eq2.5.6})
\begin{equation}
                \label{eq2.5.7}
    \begin{aligned}
   & \partial_t U = v \cdot D_x U + a^{i j} D_{v_i v_j} U \\
   &\quad+   1_{m \ge 1} \sum_{  \widetilde \alpha:  \, \widetilde \alpha < \alpha, |\widetilde \alpha| = m - 1 } c_{\widetilde \alpha}
	  D^{\widetilde \alpha}_v D^{\alpha - \widetilde \alpha + \beta}_{x} u  \quad  \text{in} \, \, (-1, 0) \times \bR^d \times B_1.
	  \end{aligned}
\end{equation}
 The above formula combined with  \eqref{eq2.5.1} and \eqref{eq2.5.4} yields
\begin{equation*}
    \|\partial_t D_v^m D_x^l u\|_{  L_2 (Q_r) } \le  N  \delta^{-\theta} (\|D_v  u\|_{  L_2 (Q_R) }
   +  \| D_x  u\|_{  L_2 (Q_R) }).
\end{equation*}
Thus, \eqref{eq2.5.8} holds.

\textbf{Step 2: $L_{\infty}$ estimate of derivatives.}
By \eqref{eq2.5.8} and the Sobolev embedding theorem, for any $l,m\ge 0$ such that $l+m\ge 1$,
\begin{equation}
        \label{eq2.5.10}
\|D^l_x D_v^{m}   u\|_{L_{\infty} (Q_{r}) } \le N  \delta^{-\theta} (\|D_v u\|_{L_2 (Q_{R}) } +  \|D_x u\|_{L_2 (Q_{R}) }).
\end{equation}
To estimate $\partial_t D_x^l D^m_v u$, we use \eqref{eq2.5.7} and \eqref{eq2.5.10}:
\begin{equation}
    \label{eq2.5.9}
\begin{aligned}
&\|\partial_t^{j}  D^l_x D_v^{m}   u\|_{L_{\infty} (Q_{r}) } & \\
&\le N  \delta^{-\theta} (\|D_v u\|_{L_2 (Q_{R}) } +  \|D_x u\|_{L_2 (Q_{R}) }),  \, \, j \in \{0, 1\}, l + m \ge 1.
\end{aligned}
\end{equation}

\textbf{Step 3: proof of \eqref{eq3.6.0}}.
Observe that
$$
    P_0 (u - v \cdot (D_v u)_{Q_1})  = \chi \quad \text{in} \, \, (-1, 0) \times \bR^d \times B_1.
$$
Then, by \eqref{eq2.5.9} and the Poincar\'e  inequality,
\begin{equation}
                                    \label{eq7.56}
    \|\partial_t^j D_x^l D_v^m  u\|_{ L_{\infty} (Q_{1/2})  } \le  N  \delta^{-\theta} (\|D^2_v  u\|_{ L_2 (Q_R) }
   +  \| D_x  u\|_{ L_2 (Q_R) }),
\end{equation}
where $j \in \{0, 1\}$ and either $m\ge 2$ or $l\ge 1$.
Finally, we denote
$$
    U_{1} = u - (1/2)v^T (D^2_v u)_{Q_1} v
$$
and observe that
\begin{align*}
      &  D^2_v  U_1  = D^2_v  u  -  ( D^2_v  u)_{Q_1}, \, \, \partial_t^j D_x^l D_v^{m+2}  U_1 = \partial_t^j D_x^l D^{m+2}_v u, \quad j+m + l\ge 1, \\
       &  P_0 U_{ 1 } (z) = \chi (t) +  a^{ i j} (t) (D_{v_i v_j} u)_{Q_1},
      \, \, \quad z \in (-1, 0) \times \bR^d \times B_1.
\end{align*}
By the above identities, the desired estimate \eqref{eq3.6.0} follows from
 \eqref{eq7.56} with $U_1$ in place of $u$.
\end{proof}

\begin{lemma}
            \label{lemma 3.8}
Invoke the assumptions of Lemma \ref{lemma 3.6} and assume, additionally, that
$u (z)  = \mathsf{u}_1 (z) + \mathsf{u}_2 (t, v)$, where
\begin{itemize}
\item[--] $\mathsf{u}_1 \in S_2 ((-4, 0) \times \bR^{2d})$ satisfies $P_0 \mathsf{u}_1 = 0$ in $(-1, 0) \times \bR^d \times B_1$,
\item[--] $\mathsf{u}_2, \partial_t \mathsf{u}_2, D^2_v \mathsf{u}_2 \in L_{2, \text{loc} }((-4, 0) \times \bR^{2d})$, and $\mathsf{u}_2$ satisfies
$$
\partial_t \mathsf{u}_2 - a^{i j} (t) D_{v_i v_j}\mathsf{u}_2  = \chi (t)\quad \text{in}\, (-1, 0) \times B_1.
$$
\end{itemize}
Then, for any $j \in \{0, 1\}$ and $l, m \ge 0$ such that $j+l+m \ge 1$, there exists $\theta = \theta (d, j, l, m) > 0$ such that
\begin{equation}
            \label{eq3.8.0}
    \begin{aligned}
   \, &  \|\partial_t^j D^l_x D_v^{m+2}   u\|_{ L_{\infty} (Q_{1/2}) }
    \\
  &  \leq N   \delta^{-\theta}  \|D_v^2 u   - (D_v^2 u)_{Q_1} \|_{L_2 (Q_1)}\\
  & \quad + N \delta^{-\theta} \sum_{k = 0}^{\infty}
	  2^{-k}	  (|(-\Delta_x)^{1/3} u - ((-\Delta_x)^{1/3} u)_{  Q_{1, 2^k}  }|^2)_{ Q_{1, 2^k}}^{ 1/2   },
\end{aligned}
\end{equation}
where $N = N (d, j,  l, m)$.
\end{lemma}

\begin{proof}
By Lemma \ref{lemma 3.6}, \eqref{eq3.6.0} holds. Furthermore,  by Lemma \ref{lemma 3.5}, \eqref{eq3.5.0} is valid with $u$ replaced with $\mathsf{u}_1$.
Note that since $\mathsf{u}_2$ is independent of $x$, \eqref{eq3.5.0} is also true  for $u$. The lemma is proved.
\end{proof}

\begin{proof}[Proof of Proposition \ref{proposition 3.3}]
We may assume that the series involving $P_0 u$  in \eqref{eq3.3.0} converges.
Denote $f = P_0 u$.
We split $u$ into the  `caloric' part and a remainder and estimate each of the terms (see Section \ref{section 1.3}). After that, we prove the desired bounds of $I_1$ and $I_2$.

    \textit{`Remainder' term.}
Let $\phi = \phi (t, v) \in C^{\infty}_0 ((t_0 - ( 2\nu r)^2, t_0  + (2\nu r)^2) \times B_{2\nu r} (v_0))$ be a function
such that $\phi = 1$ on
$
    (t_0 - (\nu r)^2, t_0) \times B_{\nu r} (v_0),
$
\begin{itemize}
    \item[--] $u_1$   be the unique $S_2 ((t_0 - (2\nu r)^2) \times \bR^{2d})$ solution to the Cauchy problem
\begin{equation}
    \label{eq3.3.4}
    P_0 u_1 (z)= f (z) \phi (t, v), \quad u (t_0 - (2\nu r)^2, \cdot) = 0
\end{equation}
(see Definition \ref{definition A.5} and Theorem \ref{theorem A.1} $(iii)$),
\item[--]   $u_2 = u_2 (t, v)$ be the unique solution in the usual parabolic Sobolev space $W^{1, 2}_2 ((t_0 - (2\nu r)^2, t_0) \times \bR^d)$  to the initial-value problem
\begin{equation}
    \label{eq3.3.5}
    \partial_t u_2 (t, v)
    - a^{i j} (t) D_{v_i v_j} u_2  (t, v) = - \chi (t) \phi (t, v), \, \,  u_2 (t_0 - (2\nu r)^2, \cdot) \equiv 0
\end{equation}
\end{itemize}
(see, for example, Theorem 2.5.2 in \cite{Kr_08}).
We set
$$
    u_{\text{rem}} (z) = u_1 (z) + u_2 (t, v).
$$

  Next, we use a scaling argument. By $\widetilde u_{\text{rem}}, \widetilde f, \widetilde \phi$, and $\widetilde P_0$ we denote the functions
    and the operator defined by \eqref{eq2.0.0} and \eqref{eq2.0.1}, respectively, with $2\nu r$ in place of $r$.
Then, by  Lemma \ref{lemma 2.0}, $\widetilde u_{\text{rem}} \in \slocxv ((-1, 0) \times \bR^{2d})$ (see Definition \ref{definition 3.0}) solves the Cauchy problem
$$
   \widetilde P_0 \widetilde  u_{\text{rem}} (z) = (2\nu r)^2 \big(\widetilde f (z) - \widetilde \chi (t) \big) \widetilde \phi (t, v), \quad \widetilde u_{\text{rem}} (-1, \cdot) \equiv 0.
$$
Furthermore, by Lemma \ref{lemma 3.2}, there exists some $\theta = \theta (d) > 0$ such that for any $R \ge 1$,
\begin{align}
\label{eq3.3.7}
(|D_v^2 \widetilde  u_{\text{rem}}|^2)_{Q_{ 1, R}}^{1/2}
&\le N (2\nu r)^2 \delta^{-\theta} \sum_{k=0}^\infty 2^{- k^2/8 } (|\widetilde  f - \widetilde \chi|^2)^{1/2}_{Q_{1, (2^{k+1}/\delta^2) R } },\\
\label{eq3.3.8}
 (|(-\Delta_x)^{1/3} \widetilde u_{\text{rem}}|^2)^{1/2}_{Q_{ 1, R } }
&\le N (2\nu r)^2 \delta^{-\theta} \sum_{k=0}^\infty 2^{-2k}(|\widetilde f - \widetilde \chi|^2)^{1/2}_{Q_{1, (2^{k +1 }/\delta^2) R }}.
\end{align}
Next, note that  for any $\varkappa, c > 0$ and $A = (-\Delta_x)^{1/3}$ or $D^2_v u$,
\begin{align*}
& (|A u_{\text{rem}}|^2)^{1/2}_{Q_{ \varkappa, c \varkappa} (z_0)}
	 = (2\nu r)^{-2}	  (|A \widetilde u_{\text{rem}}|^2)^{1/2}_{Q_{\varkappa/(2\nu r), c \varkappa/(2\nu r)}}.
 \end{align*}	
Combining \eqref{eq3.3.7} - \eqref{eq3.3.8} with the above identity, we obtain for any $R \ge 1$,
 \begin{align}
 \label{eq3.3.2}
&(|D_v^2  u_{\text{rem}}|^2)_{Q_{ 2\nu r, (2\nu r) R} (z_0)}^{1/2} \le N  \delta^{-\theta} \sum_{k=0}^\infty 2^{- k^2/8 }  F_k (R), \\
    \label{eq3.3.3}
   & (|(-\Delta_x)^{1/3} u_{\text{rem}}|^2)^{1/2}_{Q_{ 2\nu r,  (2\nu r) R} (z_0)}  \le N  \delta^{-\theta} \sum_{k=0}^\infty 2^{-2k} F_k (R),
   \end{align}
where
$$
    F_k (R) = (|f - \chi|^2)^{1/2}_{Q_{2\nu r, (2^{k +1 }/\delta^2) R (2\nu r) } (z_0)}.
$$

\textit{`Caloric' term.}
Denote $u_c = u - u_{\text{rem}} \in S_{2,\text{loc}} ((-4, 0) \times \bR^{2d})$.
Let $\overline{P}_0$ be   the operator given by \eqref{eq2.0.1} with $\nu r$ in place of $r$. For a function $h$ on $\bR^{1+2d}$,
by $\overline{h}$ we denote the function defined by \eqref{eq2.0.0} with $\nu r$ in place of $r$. Then, by Lemma \ref{lemma 2.0},
\begin{equation}
                \label{eq3.3.6}
    \overline{P}_0 \overline{u}_c (z) = (\nu r)^2 \overline{\chi} (t) \quad \text{in} \, \, (-1, 0) \times \bR^d \times B_1.
\end{equation}
Note that
\begin{itemize}
\item[--]
$\overline{u}_c (z) = \mathsf{u}_1 (z) + \mathsf{u}_2 (t, v)$, where $\mathsf{u}_1  =  \overline{u} -  \overline{u}_1$, $\mathsf{u}_2  = -  \overline{u}_2$,
 and $u_1$ and $u_2$ are defined by \eqref{eq3.3.4} and \eqref{eq3.3.5}, respectively;
\item[--]
 the conditions of  Lemma \ref{lemma 3.8} are satisfied due to \eqref{eq3.3.6} and the facts that $\mathsf{u}_1 \in S_2 ((-4, 0) \times \bR^{2d})$, and $\mathsf{u}_2 \in W^{1, 2}_2 ((-4, 0) \times \bR^d)$.
 \end{itemize}
Then, by this lemma, the bound \eqref{eq3.8.0} holds with $u$ replaced with $\overline{u}_c$.
Consequently, for any $\nu \ge 2$, we have
\begin{align}
\label{eq3.37}
  & (|D^2_v \overline{ u}_c -  (D_v^2 \overline{ u}_c)_{ Q_{1/\nu}  }|^2)^{1/2}_{ _{ Q_{1/\nu}  } }\leq  \sup_{ z_1, z_2 \in Q_{1/\nu} } |D^2_v \overline{ u}_c (z_1)  - D^2_v \overline{ u}_c (z_2)|\notag\\
  &  \leq N   \nu^{-1} \delta^{-\theta}  (|D_v^2 \overline{ u}_c   - (D_v^2 \overline{ u}_c)_{Q_1}|^2)^{1/2}_{Q_1}\\
  &\quad  + N  \nu^{-1}\delta^{-\theta} \sum_{k = 0}^{\infty}
	  2^{-k} 	  (|(-\Delta_x)^{1/3} \overline{ u}_c - ((-\Delta_x)^{1/3} \overline{ u}_c)_{  Q_{1, 2^k}  }|^2)_{ Q_{1, 2^k}}^{ 1/2   }\notag.
\end{align}
Furthermore, by \eqref{eq3.3.6} and Lemma \ref{lemma 3.5} $(i)$, we have $(-\Delta_x)^{1/3} \overline{u}_c \in S_{2, \text{loc}} ((-1, 0) \times \bR^{d} \times B_1)$, and the identity
$$
     \overline{ P}_0   (-\Delta_x)^{1/3} \overline{ u} = 0 \quad \text{in} \, \, (-1, 0) \times \bR^d \times B_1
$$
is valid.
Hence, by Lemma \ref{lemma 3.7},
\begin{align}
\label{eq3.38}
    &(| (-\Delta_x)^{1/3} \overline{ u}_c -  ((-\Delta_x)^{1/3} \overline{ u}_c)_{ Q_{1/\nu}  }|^2)^{1/2}_{ _{ Q_{1/\nu}  } }\notag\\
   & \le \sup_{ z_1, z_2 \in Q_{1/\nu} } |(-\Delta_x)^{1/3} \overline{ u}_c (z_1)  - (-\Delta_x)^{1/3}\overline{ u}_c (z_2)|\\
   & \le N  \nu^{-1} \delta^{-\theta}  (|(-\Delta_x)^{1/3} \overline{ u}_c   - ((-\Delta_x)^{1/3} \overline{u}_c)_{Q_1}|^2)^{1/2}_{Q_1}\notag.
\end{align}
 Combining \eqref{eq3.37} - \eqref{eq3.38} with the identity
 \begin{align*}
 &(|A u_c -  (A u_c)_{ Q_{ \varkappa, c \varkappa} (z_0)}|^2)^{1/2}\\
&	= (\nu r)^{-2}
\big(|A \overline{ u}_{c} - (A \overline{ u}_{c})_{ Q_{\varkappa/(2\nu r), c \varkappa/(2\nu r)} }|^2\big)^{1/2}_{ Q_{\varkappa/(2\nu r), c \varkappa/(2\nu r)} }, \quad A = (-\Delta_x)^{1/3}, D^2_v,
\end{align*}
 we obtain
\begin{align}
\label{eq3.3.12}
 & (|D^2_v  u_c -  (D_v^2 u_c)_{ Q_r (z_0) }|^2)^{1/2}_{ Q_r (z_0) }\\
&\le	N \nu^{-1} \delta^{-\theta}   (|D_v^2  u_c - (D_v^2 u_c)_{ Q_{\nu r} (z_0) }|^2)^{1/2}_{ Q_{\nu r} (z_0) }\notag\\
	& 	\,\, + N \nu^{-1} \delta^{-\theta} \sum_{k=0}^\infty 2^{-k} (|(-\Delta_x)^{1/3} u_c - ((-\Delta_x)^{1/3} u_c)_{Q_{\nu r,2^{k} \nu r } (z_0)}|^2)^{1/2}_{Q_{\nu r,2^{k} \nu r } (z_0)},\notag\\
\label{eq3.3.13}
 & (|(-\Delta_x)^{1/3}  u_c -  ((-\Delta_x)^{1/3} u_c)_{ Q_{ r} (z_0) }|^2)^{1/2}_{ Q_{r} (z_0) }\\
&\le
	N \nu^{-1} \delta^{-\theta}   (|(-\Delta_x)^{1/3}  u_c - ((-\Delta_x)^{1/3} u_c)_{ Q_{\nu r} (z_0) }|^2)^{1/2}_{ Q_{\nu r} (z_0) }\notag.
	\end{align}

\textbf{Estimate of $I_1$.}
First, note that by \eqref{eq3.3.3} with $R = 1$,
\begin{align*}
      (|(-\Delta_x)^{1/3} u_{\text{rem}}|^2)^{1/2}_{Q_r (z_0)}
&\le N  \nu^{1+2d} (|(-\Delta_x)^{1/3} u_{\text{rem}}|^2)^{1/2}_{Q_{ 2  \nu r} (z_0)}\\
&
\le N \nu^{1+2d}  \delta^{-\theta} \sum_{k=0}^\infty 2^{-2k} F_k (1)\notag.\, \,
\end{align*}
 This combined with  \eqref{eq3.3.13} and the triangle inequality give the desired estimate:
\begin{align*}
&     \big(|(-\Delta_x)^{1/3} u -  ((-\Delta_x)^{1/3} u)_{ Q_r (z_0) }|^2\big)^{1/2}_{ Q_r (z_0) }\\
&     \leq N \nu^{-1} \delta^{-\theta} (|(-\Delta_x)^{1/3}  u -  ((-\Delta_x)^{1/3} u)_{ Q_{\nu r} (z_0) }|^2)^{1/2}_{ Q_{\nu r} (z_0) }\\
 &\quad    + N  \nu^{1+2d}\delta^{-\theta}  (|(-\Delta_x)^{1/3} u_{\text{rem}}|^2)^{1/2}_{ Q_{2\nu r} (z_0) }\\
&    \leq N \delta^{-\theta}    (|(-\Delta_x)^{1/3}  u -  ((-\Delta_x)^{1/3} u)_{ Q_{\nu r} (z_0) }|^2)^{1/2}_{ Q_{\nu r} (z_0) }
+ N \nu^{1+2d} \delta^{-\theta} \sum_{k=0}^\infty 2^{-2k} F_k (1).
\end{align*}

\textbf{Estimate of $I_2$.}
By \eqref{eq3.3.2} with $R = 1$,
$$
(|D_v^2  u_{\text{rem}}|^2)_{Q_{ r} (z_0)}^{1/2}
\le N    \delta^{-\theta} \nu^{ 1+2d }\sum_{k=0}^\infty 2^{- k^2/8 } F_k (R),
$$
and hence, by the triangle inequality, we only need to estimate $I_2$ with $u$ replaced with $u_{c}$.

Next, by using  \eqref{eq3.3.12}, we get
\begin{equation}
             \label{eq3.3.14}
\begin{aligned}
 & (|D^2_v  u_c -  (D_v^2 u_c)_{ Q_r (z_0) }|^2)^{1/2}_{ Q_r (z_0) }\\
&\le
	N \nu^{-1} \delta^{-\theta}   (|D_v^2  u - (D_v^2 u)_{ Q_{ \nu r} (z_0) }|^2)^{1/2}_{ Q_{\nu r} (z_0) } \\
&
	\quad + N \nu^{-1} \delta^{-\theta} \sum_{k=0}^\infty 2^{-k} (|(-\Delta_x)^{1/3}u - ((-\Delta_x)^{1/3}u)_{Q_{\nu r,2^{k} \nu r } (z_0)}|^2)^{1/2}_{Q_{\nu r,2^{k} \nu r } (z_0)}\\
&\quad	+ N  \nu^{-1} \delta^{-\theta} (J_1+J_2),\\
 \end{aligned}
 \end{equation}
 where
 $$
    J_1 = (|D_v^2  u_{\text{rem}}|^2)^{1/2}_{ Q_{\nu r} (z_0) },
    \quad J_2 = \sum_{k=0}^\infty 2^{-k} (|(-\Delta_x)^{1/3} u_{\text{rem}}|^2)^{1/2}_{Q_{\nu r,2^{k} \nu r } (z_0)}.
 $$
 The term $J_1$ is estimated in \eqref{eq3.3.2} with $R = 1$. Furthermore, using \eqref{eq3.3.3} with $R = 2^k$ gives
\begin{align*}
&J_2 \le   N (d)  \sum_{l  =0}^\infty 2^{  -2 l  }   \sum_{k   = 0}^\infty 2^{-k} F_l (2^k).
\end{align*}
Noticing that $F_l (2^k) = F_{l+k} (1)$ and changing the index of summation $k \to k+l$, we obtain
\begin{equation}
               \label{eq3.3.15}
        J_2 \le N (d) \sum_{k   = 0}^\infty 2^{-k} F_k (1).
\end{equation}
Combining the  inequalities \eqref{eq3.3.14} - \eqref{eq3.3.15}, \eqref{eq3.3.2}, we prove the estimate of $I_{2}$ with $u$ replaced with $u_c$. As was mentioned above, this implies the desired bound of $I_2$.
\end{proof}

\section{Proof of Theorem \ref{theorem 1.5}}
                                \label{section 4}
In this section, we first show a few intermediate results and then prove Theorem \ref{theorem 1.5}.
\begin{lemma}
                \label{lemma 4.1}
For any $\alpha \in (0, 1)$ and $u \in \mathbb{C}^{\alpha} (\bR^{1+2d}_T) \cap S_2 (\bR^{1+2d}_T)$ (see \eqref{1.7} and \eqref{eq1.11}),
we have
\begin{equation}
            \label{eq4.1.0}
\begin{aligned}
 & [D^2_v u] +  [(-\Delta_x)^{1/3} u] \\
&\le
  N  \delta^{-\theta} \big([P u]_{ L_{\infty} C^{\alpha/3, \alpha}_{x, v} (\bR^{1+2d}_T)} + \|u\|_{ L_{\infty} (\bR^{1+2d}_T) }\big),
\end{aligned}
\end{equation}
where $N = N (d, \alpha, K)>0$ and $\theta = \theta (d, \alpha) > 0$.
\end{lemma}

\begin{proof}
The idea is to perturb the mean-oscillation estimates  in Proposition \ref{proposition 3.3} to bound the Campanato's  seminorms  (see \eqref{eqB.1.10}) of $(-\Delta_x)^{1/3} u$ and $D^2_v u$.
In this  proof, if not specified, we assume $N = N (d, \alpha, K)$.

\textbf{Step 1: freezing the coefficients.}
We fix some $z_0 \in \overline{\bR^{1+2d}_T}$. For any function $h$ on $\bR^{1+2d}_T$,  denote
$$
 \overline{h} (t) = h (t, x_0 - (t-t_0)v_0, v_0),\quad
 \mathcal{P}_0 = \partial_t  - v \cdot D_x - \overline{a}^{i j} (t) D_{v_i v_j}.
$$
By the identity
$$
   \mathcal{P}_0 u -   \overline{ P u} = P u  - \overline{P u} - (\overline{a}^{i j} - a^{i j}) D_{v_i v_j} u
$$
and Proposition \ref{proposition 3.3} with $a$ replaced with $\overline{a}$ and $\chi = \overline{P u}$, there exists $\theta_0 = \theta_0 (d) > 0$ such that
\begin{align}
            \label{eq4.1}
 &   \bigg(|(-\Delta_x)^{1/3} u -  ((-\Delta_x)^{1/3} u)_{ Q_r (z_0) }|^2\bigg)^{1/2}_{Q_r (z_0)}\\
&\leq N  \nu^{-1} \delta^{-\theta_0}
   (|(-\Delta_x)^{1/3} u - ((-\Delta_x)^{1/3} u)_{Q_{\nu r} (z_0)}|^2)^{1/2}_{ Q_{\nu r } (z_0)}\notag\\
    &\quad + N  \nu^{1+2d}  \delta^{-\theta_0} (J_1+J_2),\notag
    \\
                \label{eq4.2}
    &	\bigg(|D_v^2 u -  (D_v^2 u)_{ Q_r (z_0) }|^2\bigg)^{1/2}_{ Q_r (z_0) } \\
&
  \leq
N   \nu^{-1} \delta^{-\theta_0}
    (|D_v^2 u - (D_v^2 u)_{Q_{\nu r} (z_0)}|^2)^{1/2}_{ Q_{\nu r} (z_0)  }\notag\\
&\quad    + N \nu^{-1} \delta^{-\theta_0}
    \sum_{k = 0}^{\infty}
	 2^{-k}
	 (|(-\Delta_x)^{1/3} u - ((-\Delta_x)^{1/3} u)_{ Q_{\nu r, 2^k \nu r} (z_0)}|^2)^{1/2}_{ Q_{\nu r, 2^k \nu r} (z_0)}\notag\\
	 &\quad  + N  \nu^{1+2d}  \delta^{-\theta_0}  (J_1+J_2),\notag
	 \end{align}
	 where $N = N (d)$, and
	 \begin{align*}
	    &J_1 = \sum_{k=0}^\infty 2^{-k}
	    \big(|P u - \overline{P u}|^2\big)^{1/2}_{Q_{ 2\nu r, (2^{k+1}/\delta^2) (2\nu r)} (z_0)},\\
	    & J_2 = \sum_{k=0}^\infty 2^{-k}
	    \big(|(a^{i j} - \overline{a}^{i j}) D_{v_i v_j} u|^2\big)^{1/2}_{Q_{ 2\nu r, (2^{k+1}/\delta^2) (2\nu r)} (z_0)}.
	 \end{align*}

Next, by Lemma \ref{lemma B.4} $(i)$ and Assumption \ref{assumption 1.2},
\begin{align}
\label{eq4.3}
    J_1 &\le N [P u]_{L_{\infty} C^{\alpha/3, \alpha}_{x, v} (\bR^{1+2d}_T) }  \delta^{-2 \alpha} (\nu r)^{\alpha},\\
   \label{eq4.4}
    J_2 &\le N \delta^{-2 \alpha} [a]_{ _{L_{\infty} C^{\alpha/3, \alpha}_{x, v} (\bR^{1+2d}_T) } } \|D^2_v u\|_{ L_{\infty} (\bR^{1+2d}_T) } (\nu r)^{\alpha}\\
   & \le N \delta^{-1-2\alpha} \|D^2_v u\|_{ L_{\infty} (\bR^{1+2d}_T) }  (\nu r)^{\alpha}\notag.
\end{align}

	 \textbf{Step 2: Campanato type argument.}
	\textit{Estimate of $(-\Delta_x)^{1/3} u$.} Denote
	 	 \begin{equation}
	 	                \label{eq4.5}
	    \psi_1 (r) = \bigg(\int_{Q_r (z_0)} |(-\Delta_x)^{1/3} u - ((-\Delta_x)^{1/3} u)_{ Q_r (z_0) }|^2 \, dz\bigg)^{1/2}.
	  \end{equation}
	   Note that $\psi_1$ is a nondecreasing function bounded by $\|(-\Delta_x)^{1/3} u\|_{ L_2 (\bR^{1+2d}_T) }$, which is finite due to Corollary \ref{corollary A.4} and the fact that $u \in S_2 (\bR^{1+2d}_T)$.
	 Multiplying \eqref{eq4.1} by $|Q_r|^{1/2} = c_d r^{1+2d}$ and  using \eqref{eq4.3}  - \eqref{eq4.4} give
	 \begin{align*}
	    \psi_1 (r) \le N \delta^{-\theta_0} \nu^{-2-2d} \psi_1 (\nu r) + N  \delta^{-\theta} (\nu r)^{1+2d+\alpha} (A+B),
	 \end{align*}
	 where $\theta =\theta (d) > 0$,
	 $$
	    A = [P u]_{L_{\infty} C^{\alpha/3, \alpha}_{x, v} (\bR^{1+2d}_T) }, \quad B = \|D_v^2 u\|_{ L_2 (\bR^{1+2d}_T) }.
	 $$
Let $\tilde \alpha=(1+\alpha)/2\in (\alpha,1)$.	 Taking $\nu$ large so that
$N  \nu^{\tilde\alpha-1} \delta^{-\theta_0} = 1$,     we have
    \begin{align*}
        \psi_1 (r) \le  \nu^{-(1+2d+\tilde\alpha)} \psi_1 (\nu r) + N \delta^{-\theta} (\nu r)^{1+2d+\alpha} (A+B).
    \end{align*}
By a standard iteration argument (cf. Lemma 5.13 of \cite{GM_12}),  we get
    $$
       \psi_1 (r) \le   N \delta^{-\theta}  r^{1+2d+\alpha} (A+B).
    $$
    The latter combined with  Lemma \ref{lemma 2.1} yields
    \begin{equation}
                \label{eq4.6}
        [(-\Delta_x)^{1/3} u]_{\kC^{\alpha} (\bR^{1+2d}_T)} \le  N \delta^{-\theta} (A+B).
        \end{equation}

    \textit{Estimate of $D^2_v u$.} Let $\psi_2$ be the function defined by \eqref{eq4.5} with  $(-\Delta_x)^{1/3} u$ replaced with $D^2_v u$.
    Note that by Lemma \ref{lemma B.4} $(ii)$ and \eqref{eq4.6}, the second term on the right-hand side of \eqref{eq4.2} is bounded by
    $$
        N \nu^{-1+\alpha} \delta^{-\theta} r^{\alpha}  [(-\Delta_x)^{1/3} u]_{\kC^{\alpha} (\bR^{1+2d}_T)} \le N \delta^{-\theta} (\nu r)^{\alpha} (A+B).
    $$
    Then, multiplying \eqref{eq4.2} by $|Q_r|^{1/2}$ and using the above inequality combined with \eqref{eq4.3} - \eqref{eq4.4}, we get
    $$
        \psi_2 (r) \le N \delta^{-\theta_0} \nu^{-2-2d} \psi_2 (\nu r) + N  \delta^{-\theta} (\nu r)^{1+2d+\alpha} (A+B).
    $$
    As above, we conclude that
$$
    [D^2_v u]_{\kC^{\alpha} (\bR^{1+2d}_T)} \le  N \delta^{-\theta} (A+B).
$$
Adding the last inequality to \eqref{eq4.6} gives
$$
     [(-\Delta_x)^{1/3} u]_{\kC^{\alpha} (\bR^{1+2d}_T)} +  [D^2_v u]_{\kC^{\alpha} (\bR^{1+2d}_T)} \le  N \delta^{-\theta} (A+B).
$$
By using the interpolation inequality in Remark \ref{remark 1.2}, we may replace $B$  with $\|u\|_{L_{\infty} (\bR^{1+2d}_T)}$ in the last estimate, which proves \eqref{eq4.1.0}.
\end{proof}

\begin{lemma}
        \label{lemma 4.2}
For any $\alpha \in (0, 1)$, there exists $\lambda_0$  as in \eqref{eq1.5.0}
such that for any $\lambda \ge \lambda_0$ and $u \in \mathbb{C}^{\alpha} (\bR^{1+2d}_T) \cap S_2 (\bR^{1+2d}_T)$,
\eqref{eq1.5.1} holds.
\end{lemma}

\begin{proof}
\textbf{Step 1: case when $b \equiv 0$, $c \equiv 0$.}
We use S. Agmon's method to  derive \eqref{eq1.5.1} from \eqref{eq4.1.0}. In particular, by this method, we are able to prove the bounds of $D^k_v u, k = 0, 1, 2$.
These estimates imply the validity of \eqref{eq1.5.1} for
$(-\Delta_x)^{1/3} u$ and $\partial_t u - v \cdot D_x u$.

\textit{Agmon's method} (cf. Lemma 6.3.8 in \cite{Kr_08}).
Denote
 \begin{align*}
	& 
	\widehat x = (x_1, \ldots, x_{d+1}),
	 \quad  \widehat v = (v_1, \ldots,  v_{d+1}),
	 \quad  \widehat z = (t, \hat x, \hat v),\\
& 
\widehat P (\widehat z) = \partial_t - \sum_{i = 1}^{d+1} v_i D_{x_i} - \sum_{i, j = 1}^{d} a^{i j} (z) D_{v_i v_j}  - D_{v_{d+1} v_{d+1}}.
\end{align*}

Let  $\zeta$ be a smooth cutoff function on $\bR$ such that  $\zeta (y) = 1$ for $y\in (-1,1)$ and
 denote for $k\ge 1$,
$$
	\widehat U (\widehat z) = u (z)  \cos (\lambda v_{d+1} + \pi/4) \zeta (v_{d+1}/k) \zeta (x_{d+1}/k^3).
$$
We choose such $\widehat U$  due to the following technical reasons:
\begin{itemize}
\item[--]
$
    \widehat U   \in \mathbb{C}^{\alpha} (\bR^{1+2(d+1)}_T) \cap S_2 (\bR^{1+2(d+1)}_T),
$
so that Lemma \ref{lemma 4.1} can be applied to $\widehat U$.

\item[--] $\zeta (x_{d+1}/k^3) \zeta (v_{d+1}/k)$ and all its partial derivatives are of class $\kC^{\alpha} (\bR^{1+2(d+1)}_T)$ (see Remark \ref{remark 1.1}).
This fact is used in the estimate \eqref{eq4.14} below.
\end{itemize}
Computing directly, we get
\begin{align}
                    \label{eq4.7}
&\lambda^2 \widehat U (\widehat z) =	\lambda^2 u (z) \zeta (v_{d+1}/k) \cos (\lambda v_{d+1} + \pi/4) \zeta (x_{d+1}/k^3) \\
 &
= - D_{v_{d+1} v_{d+1}} \widehat U (\widehat z)+ u (z)   \zeta (x_{d+1}/k^3) \big(k^{-2}\zeta'' (v_{d+1}/k) \cos (\lambda v_{d+1} + \pi/4) \notag\\
&\quad - 2 \lambda  k^{-1}\zeta' (v_{d+1}/k) \sin (\lambda v_{d+1} + \pi/4)\big),\notag\\
          \label{eq4.8}
J&:=\lambda D_{v_i} u (z) \sin (\lambda v_{d+1} + \pi/4) \zeta (v_{d+1}/k) \zeta (x_{d+1}/k^3) \\
 &= - D_{v_{d+1} v_i }  \widehat U (\widehat z)
+ k^{-1} D_{v_i} u (z) \zeta' (v_{d+1}/k) \zeta (x_{d+1}/k^3) \cos (\lambda v_{d+1} + \pi/4)  \notag.
\end{align}	
We will extract the estimates of $u$ and $D_v u$ from the above identities.

\textit{Estimate of $u,  D_{v} u$.}
 By the product rule inequality in Remark \ref{remark 1.4}, for any $h_1, h_2   \in \kC^{\alpha}  (\bR^{1+2d}_T)$ or $L_{\infty} C^{\alpha/3, \alpha}_{x, v} (\bR^{1+2d}_T)$,
 and  any $\lambda > 1$, we have
  \begin{equation}
            \label{eq4.11}
   [h_1 (\lambda^2  \cdot, \lambda^3 \cdot, \lambda \cdot) h_2]_X \le N (h_1 , \alpha) ( [h_2]_X + \lambda^{\alpha} \|h_2\|_{ L_{\infty} (\bR^{1+2d}_T) }),
\end{equation}
where $X$ is either $\kC^{\alpha} (\bR^{1+2d}_T)$ or $L_{\infty} C^{\alpha/3, \alpha}_{x, v} (\bR^{1+2d}_T)$.
Furthermore, for $k, \lambda  \ge 1$, one has
\begin{equation}
                \label{eq4.13}
        N_1 \lambda^{\alpha} \le [\cos (\lambda \cdot + \pi/4) \zeta (\cdot/k)]_{ C^{\alpha} (\bR) } \le N_1^{-1} \lambda^{\alpha}
\end{equation}
and a similar bound with sine instead of cosine,
where $N_1 = N_1 (\alpha, \zeta)$.
Combining \eqref{eq4.7} - \eqref{eq4.13} gives
\begin{align}
\label{eq4.14}
  &   \lambda^2 [u]_{ \kC^{\alpha} (\bR^{1+2d}_T) } + \lambda [D_v u]_{ \kC^{\alpha} (\bR^{1+2d}_T) }  + \lambda^{\alpha} \|\lambda^{2} |u| + \lambda |D_v u|\|_{ L_{\infty} (\bR^{1+2d}_T) }  \\
  &\le  N \lambda^2 [\widehat U]_{ \kC^{\alpha} (\bR^{1+2(d+1)}_T) }
  +   N  [J]_{ \kC^{\alpha} (\bR^{1+2(d+1)}_T) }
  \notag\\
  & \le  N [D^2_{\widehat v} \,  \widehat U]_{ \kC^{\alpha} (\bR^{1+2(d+1)}_T) }
  + N \lambda^{\alpha} k^{-1} (\|u\|_{\kC^{\alpha} (\bR^{1+2d}_T) } + \|D_v u\|_{\kC^{\alpha} (\bR^{1+2d}_T) })\notag,
   \end{align}
where $N = N (d, \alpha)$.

\textit{Estimate of $D^2_{\widehat v} \widehat U$.}
  Since $\widehat U \in \mathbb{C}^{\alpha} (\bR^{1+2(d+1)}_T) \cap S_2 (\bR^{1+2(d+1)}_T)$, by Lemma \ref{lemma 4.1},
\begin{equation}
            \label{eq4.10}
    \begin{aligned}
         &   [D^2_{\widehat v} \,  \widehat U]_{ \kC^{\alpha} (\bR^{1+2(d+1)}_T) } \\
   & \le N \delta^{-\theta} \big([\widehat P \widehat U (\widehat z)]_{ L_{\infty} C^{\alpha/3, \alpha}_{x, v} (\bR^{1+2(d+1)}_T)} + \|\widehat U\|_{ L_{\infty} (\bR^{1+2(d+1)}_T) }),
        \end{aligned}
    \end{equation}
where
\begin{align}
         \label{eq4.9}
&   \widehat P \widehat U (\widehat z)  = \zeta (v_{d+1}/k)  \zeta (x_{d+1}/k^3) \cos (\lambda v_{d+1} + \pi/4)
(P u (z) + \lambda^2 u (z)) \\
	&-  u (z) \zeta (x_{d+1}/k^3) \big(k^{-2}\zeta'' (v_{d+1}/k) \cos (\lambda v_{d+1} + \pi/4) \notag\\
&\quad - 2 k^{-1}\lambda \zeta' (v_{d+1}/k) \sin (\lambda v_{d+1}+ \pi/4)\big)\notag\\
	&  - u (z) (v_{d+1} \zeta (v_{d+1}/k)  k^{-3} \zeta' (x_{d+1}/k^3)) \cos (\lambda^{1/2} v_{d+1} + \pi/4)\notag.
   \end{align}
By  \eqref{eq4.10} - \eqref{eq4.9} and \eqref{eq4.11}, for $\lambda , k \ge 1$,
  \begin{equation}
              \label{eq4.12}
  \begin{aligned}
  & \lambda^{\alpha} \|D_v^2 u\|_{ L_{\infty} (\bR^{1+2d}_T) }  + [D^2_v u]_{ \kC^{\alpha} (\bR^{1+2d}_T) } \le N   [D^2_{\widehat v} \,  \widehat U]_{ \kC^{\alpha} (\bR^{1+2(d+1)}_T) } \\
 &   \le
 N \delta^{-\theta} ([P u + \lambda^2 u]_{ L_{\infty}  C^{\alpha/3, \alpha}_{x, v} (\bR^{1+2d}_T)}+\|u\|_{L_\infty(\bR^{1+2d}_T)})\\
  &\quad +   N \delta^{-\theta} \lambda^{\alpha} (\|P u + \lambda^2 u\|_{ L_{\infty}  (\bR^{1+2d}_T)} + k^{-1}\|u\|_{ L_{\infty}  C^{\alpha/3, \alpha}_{x, v} (\bR^{  1+2d}_T)}),
\end{aligned}
\end{equation}
where $N = N (d, \alpha) > 0$.

Combining \eqref{eq4.14} with \eqref{eq4.12} and sending $k\to \infty$, we get
\begin{equation*}
\begin{aligned}
    &     \lambda^2 [u]_{ \kC^{\alpha} (\bR^{1+2d}_T) } + \lambda [D_v u]_{ \kC^{\alpha} (\bR^{1+2d}_T) } + [D^2_v u]_{ \kC^{\alpha} (\bR^{1+2d}_T) } \\
&
  \quad  +\lambda^{2+\alpha} \|u\|_{L_\infty(\bR^{1+2d}_T)} + \lambda^{1+\alpha} \|D_v u\|_{L_\infty(\bR^{1+2d}_T)}
    + \lambda^{\alpha} \|D_v^2 u\|_{L_\infty(\bR^{1+2d}_T)} \\
&  \le N  \delta^{-\theta} ( [P u + \lambda^2 u]_{ L_{\infty}  C^{\alpha/3, \alpha}_{x, v} (\bR^{1+2d}_T)}+\|u\|_{L_\infty(\bR^{1+2d}_T)})\\
&\quad +  N  \delta^{-\theta} \lambda^{\alpha}
 \|P u + \lambda^2 u\|_{ L_{\infty}  (\bR^{1+2d}_T)}\notag.
\end{aligned}
\end{equation*}
 By taking $\lambda_0\ge \max\{1,  2N\delta^{-\theta}\}$, we obtain the bounds for $u, D_v u$, and $D^2_v u$.

\textit{Estimates of the transport term.}
By the identity
\begin{equation}
            \label{eq4.17}
    \partial_t u - v \cdot D_x u = (P + \lambda^2) u - a^{i j} D_{v_i v_j} u  - \lambda^2 u
\end{equation}
and Assumptions \ref{assumption 1.1}   - \ref{assumption 1.2}, and the product rule inequality, we get
\begin{align*}
   & [\partial_t u - v \cdot D_x u]_{ L_{\infty} C^{\alpha/3, \alpha}_{x, v}  (\bR^{1+2d}_T) } \le  [(P + \lambda^2) u]_{ L_{\infty} C^{\alpha/3, \alpha}_{x, v}  (\bR^{1+2d}_T) } \\
   & +  N \delta^{-1} \|D^2_v u\|_{ L_{\infty} C^{\alpha/3, \alpha}_{x, v}  (\bR^{1+2d}_T) } + \lambda^2 [u]_{ L_{\infty} C^{\alpha/3, \alpha}_{x, v}  (\bR^{1+2d}_T) },
\end{align*}
and the right-hand is bounded by that of \eqref{eq1.5.1}.
Similarly, we can bound the $L_{\infty}$ norm of the transport term.

\textit{Estimates of $(-\Delta_x)^{1/3} u$ and the $C^{(2+\alpha)/3}_x$ seminorm.}
First,  due to Lemma \ref{lemma 4.1} and  the estimates of $u$ in \eqref{eq1.5.1}, we get
\begin{align*}
   &  [(-\Delta_x)^{1/3} u]_{ \kC^{\alpha} (\bR^{1+2d}_T) }\\
   & \le N  \delta^{-\theta} \big([P u + \lambda^2 u]_{ L_{\infty}  C^{\alpha/3, \alpha}_{x, v} (\bR^{1+2d}_T)} + \lambda^2 [u]_{ L_{\infty} C^{\alpha/3, \alpha}_{x, v} (\bR^{1+2d}_T) } + \|u\|_{L_{\infty} (\bR^{1+2d}_T)}\big)\\
   &\le N  \delta^{-\theta} \big([P u + \lambda^2 u]_{ L_{\infty}  C^{\alpha/3, \alpha}_{x, v} (\bR^{1+2d}_T)} + \lambda^{\alpha}\|P u + \lambda^2 u\|_{ L_{\infty} (\bR^{1+2d}_T) }\big).
\end{align*}

 Next, by a  mollification argument, we have the interpolation inequality
\begin{equation}
\label{eq1.9.1}
\begin{aligned}
   & \|(-\Delta_x)^{1/3} u\|_{ L_{\infty} (\bR^{1+2d}_T) }  \le N (d, \alpha) \big(\varepsilon^{\alpha} [(-\Delta_x)^{1/3} u]_{  \kC^{\alpha} (\bR^{1+2d}_T) } \\
   & +   \varepsilon^{-2} \|u\|_{ L_{\infty} (\bR^{1+2d}_T) }\big), \, \,  \forall \,\varepsilon > 0,
\end{aligned}
\end{equation}
so that  the term $\lambda^{\alpha} \|(-\Delta_x)^{1/3} u\|_{ L_{\infty} (\bR^{1+2d}_T) }$ is bounded by the right-hand side of \eqref{eq1.5.1}.
Furthermore, by using the fact that the operator
$$
(1+ (-\Delta_x)^{1/3})^{-1}: C^{\alpha/3} (\bR^d) \to C^{(2+\alpha)/3} (\bR^d)
$$
is bounded (see, for example Theorem 1.3 in \cite{DK_13}) and a scaling argument, we conclude that
$$
    \text{ sup}_{ (t, v) \in \bR^{1+d}_T }  [u (t, \cdot, v)]_{ C^{(2+\alpha/3)} (\bR^d) }
$$
is also bounded by the right-hand side of \eqref{eq1.5.1}.

\textbf{Step 2: adding the lower-order terms.} By using \eqref{eq1.5.1} and the triangle inequality, we obtain \eqref{eq1.5.1} with the right-hand side replaced with
\begin{align*}
  & N \delta^{-\theta}  [P u + b \cdot D_v u + (c + \lambda^2) u]_X \\
  &\quad + N \delta^{-\theta}  \lambda^{\alpha} \big(\|P u + b \cdot D_v u + (c + \lambda^2) u\|_{ L_{\infty} (\bR^{1+2d}_T) }
  + \|b \cdot D_v u\|_{ X} + \|c u\|_{ X }\big),
\end{align*}
where $X = L_{\infty} C^{\alpha/3, \alpha}_{x, v} (\bR^{1+2d}_T)$.
By the product rule inequality (see Remark \ref{remark 1.4}) and Assumption \ref{assumption 1.3},
\begin{equation}
            \label{eq4.16}
    \|b \cdot D_v u\|_{X} + \|c u\|_{X}  \le  L (\|D_v u\|_{X} + \|u\|_{X}).
\end{equation}
For sufficiently large  $\lambda \ge \lambda_0$ with $\lambda_0$ as in \eqref{eq1.5.0}, the terms on the right-hand side of  \eqref{eq4.16} can be absorbed into the left-hand side of \eqref{eq1.5.1}.
\end{proof}

\begin{proof}[Proof of Theorem \ref{theorem 1.5}] We prove the assertions in the following order: $(iii)$, $(ii)$, $(iv)$, and $(i)$. In particular, we will see that  $(ii)$ is an immediate corollary of $(iii)$.

\textbf{Proof of $(iii)$ and $(ii)$.}
\textit{Uniqueness.} We only need to show that in the case when $f \equiv 0$, any solution $u$ of class $C^{2, \alpha} (\bR^{1+2d}_T)$  must be identically $0$.  Let $\phi \in C^{\infty}_0 (\bR^{1+2d})$ be a function such that $\phi  = 1$ on $\widetilde Q_1$ and denote $\phi_n (z) = \phi (t/n^2, x/n^3, v/n)$.
Then, $u_n: = u \phi_n \in S_2 (\bR^{1+2d})$ satisfies
\begin{align*}
  &  P u_n + b \cdot D_v u_n + (c+\lambda^2) u_n\\
  & = u P \phi_n  - 2 (a D_v \phi_n) \cdot D_v u + (b \cdot D_v \phi_n) u =: f_n.
\end{align*}
Then by Lemma \ref{lemma 4.2} and the product rule inequality in Remark \ref{remark 1.4},
for any $\lambda \ge \lambda_0$,
\begin{align*}
    \|u \phi_n\|_{ L_{\infty} (\bR^{1+2d}_T) } \le N  \|f_n\|
    \le N n^{-1} (\|u\|+\|D_v u\|),
\end{align*}
where $\|\cdot\|$ is the $L_{\infty} C^{\alpha/3, \alpha}_{x, v} (\bR^{1+2d}_T)$ norm, and
$N = N (d, \alpha, K, L, \delta, \lambda)$.
Passing to the limit as $n \to \infty$ in the above inequality gives $u \equiv 0$.

\textit{Existence.}
Proof by  a compactness argument.
Let $\eta = \eta (x, v) \in C^{\infty}_0 (\bR^{2d}), \xi \in C^{\infty}_0 (\bR^{1+2d}_T)$ be  functions such that $\int \eta \, dx dv = 1$,
and $\xi (z) \in [0, 1] \,  \forall z$,  $\xi  = 1$ on $\widetilde Q_1$, and denote for $n \ge 1$,
\begin{align*}
& \eta_n (x, v) = n^{4d} \eta (n^3 x, n v), \quad \xi_n (z) = \xi (t/n^2, x/n^3, v/n),\\
&    h_n = h \ast \eta_n, \quad \text{where } h = a, b, c,  \\
& f_n  = (f \ast \eta_n) \xi_n.
\end{align*}
 Note that $a_n, b_n, c_n, f_n$ satisfy the assumptions of Corollary \ref{corollary A.2}, and
furthermore, by the product rule inequality (see Remark \ref{remark 1.4}),
$$
    [f_n]_{L_{\infty} C^{\alpha/3, \alpha}_{x, v} (\bR^{1+2d}_T)}  \le [f]_{L_{\infty} C^{\alpha/3, \alpha}_{x, v} (\bR^{1+2d}_T)}
    + N (\xi) n^{-\alpha} \|f\|_{ L_{\infty} (\bR^{1+2d}_T) }.
$$
Hence,
by Corollary \ref{corollary A.2}, the equation
$$
    P u_n + b \cdot D_v u_n + (c + \lambda^2) u_n  = f_n
$$
has a unique solution $\mathbb{C}^{2, \alpha} (\bR^{1+2d}_T) \cap S_2 (\bR^{1+2d}_T)$.
Then, by Lemma \ref{lemma 4.2}
there exists
$\lambda_0$ as in  \eqref{eq1.5.0}
such that for any $\lambda \ge \lambda_0$,
\begin{equation}
                \label{eq1.5.7}
\begin{aligned}
  &  \lambda^{2+\alpha} \|u_n\|_{ L_{\infty} (\bR^{1+2d}_T) }   + \lambda^{2}  [u_n] +  \lambda^{1+\alpha} \|D_v u_n\|_{ L_{\infty} (\bR^{1+2d}_T) } \\
  &\quad + \lambda [D_v u_n]
   + \lambda^{\alpha} \||D^2_v u_n| + |(-\Delta_x)^{1/3} u_n|\|_{ L_{\infty} (\bR^{1+2d}_T) } \\
  &\quad + [D^2_v u_{ n }] + [(-\Delta_x)^{1/3} u_{ n}]
  + \sup_{ (t, v) \in \bR^{1+d}_T} [u_n (t, \cdot, v)]_{ C^{(2+\alpha)/3} (\bR^d) } \\
  &  \le  N  \delta^{-\theta}  \big([f_n]_{L_{\infty} C^{\alpha/3, \alpha}_{x, v} (\bR^{1+2d}_T)}
  + \lambda^{\alpha} \|f_n\|_{  L_{\infty} (\bR^{1+2d}_T) }\big)\\
  & \le  N  \delta^{-\theta}  \big([f]_{L_{\infty} C^{\alpha/3, \alpha}_{x, v} (\bR^{1+2d}_T)}
  + (\lambda^{\alpha} +  n^{-\alpha}) \|f\|_{  L_{\infty} (\bR^{1+2d}_T) }\big),
  \end{aligned}
\end{equation}
where $[\, \cdot\, ]$ is the $\kC^{\alpha} (\bR^{1+2d}_T)$ seminorm and $N = N (d, \alpha, K)$.

Using the Arzela-Ascoli theorem and Cantor's diagonal argument, from \eqref{eq1.5.7} we conclude that there exists $u \in \kC^{2, \alpha} (\bR^{1+2d}_T)$ solving \eqref{eq1.5.4}, and,  furthermore,   \eqref{eq1.5.1} holds  with $P u + b \cdot D_v u + (c+\lambda^2) u$ replaced with $f$ for all the terms on the left-hand side excluding the transport term.
The latter is estimated  as in the proof of Lemma \ref{lemma 4.2} (see p. \pageref{eq4.17}) by using Eq. \eqref{1.1}.
Thus, $(iii)$ is true.
Moreover, the a priori estimate proved for the solution of \eqref{1.1} combined with the uniqueness part implies the validity of the assertion $(ii)$.

\textbf{Proof of $(iv)$.} The assertion is derived in a standard way by using $(ii)$ and an exponential weight in the temporal variable.

\textbf{Proof of $(i)$.} Note that \eqref{eq1.5.2} does not follow from \eqref{eq1.5.1} by  setting $\lambda = \lambda_0$ in \eqref{eq1.5.1}.
Indeed, the latter gives an estimate weaker than \eqref{eq1.5.1} since it has extra terms involving $[u]_{L_{\infty} C^{\alpha/3, \alpha}_{x, v} (\bR^{1+2d}_T)}$
and $\|P u + b \cdot D_v u+ c u\|_{L_{\infty} (\bR^{1+2d}_T)}$. To avoid this issue, we prove that Lemma \ref{lemma 4.1} still holds if  $u \in \kC^{2,\alpha} (\bR^{1+2d}_T)$.

\textit{Step 1.} We claim that Proposition \ref{proposition 3.3} still  holds if $u \in \kC^{2, \alpha} (\bR^{1+2d}_T)$.
Instead of repeating the argument, we list some places therein that need to be modified.
\begin{itemize}
    \item Note that $f  = P_0 u \in L_{\infty} C^{\alpha/3, \alpha}_{x, v} (\bR^{1+2d}_T)$ and that by Theorem \ref{theorem 1.5} $(iv)$, the Cauchy problem \eqref{eq3.3.4} has a unique solution  $u_1 \in \kC^{2,\alpha} ((t_0 - (2\nu r)^2 , t_0) \times \bR^{2d})$.\\

    \item We need to show that Lemma \ref{lemma 3.5} still holds for $u \in \kC^{2, \alpha} ((-4, 0) \times \bR^{2d})$, which would also imply that Lemma \ref{lemma 3.8} is valid for such $u$. First, by Theorem \ref{theorem 1.5}, $(-\Delta_x)^{1/3} u \in \kC^{\alpha} ((-1, 0)\times \bR^{2d})$ (cf. the proof of Corollary \ref{corollary A.4}), and then, due to  Lemma \ref{lemma B.4} $(ii)$, the series on the right-hand side of \eqref{eq3.5.0} converges. Second, it follows from  $u \in \kC^{2, \alpha} ((-4, 0) \times \bR^{2d})$ that \eqref{eq2.1.1} holds. The rest of the argument is the same as that of Lemma \ref{lemma 3.5}.
\end{itemize}

\textit{Step 2: proof of \eqref{eq1.5.2}.}
The argument is the same as that of Lemma \ref{lemma 4.1} with one modification: we do not need to use an iteration argument  to conclude  that $(-\Delta_x)^{1/3} u, D^2_v u \in \kC^{\alpha} (\bR^{1+2d}_T)$ (see Step 2 therein) since the latter follows from the definition of $\kC^{2, \alpha} (\bR^{1+2d}_T)$. Furthermore, multiplying \eqref{eq4.1} - \eqref{eq4.2} by $r^{-\alpha}$, taking supremum over $r > 0$, and then taking $\nu$ sufficiently large, we conclude that \eqref{eq1.5.2} holds for  $(-\Delta_x)^{1/3} u$ and $D^2_v u$.
The  $C^{(2+\alpha)/3}_x$ seminorm of $u$ is estimated in the same way as in the proof of Lemma \ref{lemma 4.2} (see  p. \pageref{eq1.9.1}).
Finally, as in the proof of Lemma \ref{lemma 4.2}, we extract the estimate of the transport term from the identity \eqref{eq4.17} by the product rule inequality and the standard interpolation inequality.
\end{proof}

\section{Proof of Corollaries  \ref{corollary 1.6} - \ref{corollary 1.7}}
                        \label{section 5}

\begin{proof}[Proof of Corollary \ref{corollary 1.6}]
By a scaling argument, it suffices to prove the estimate in the case when $\varepsilon = 1.$

By \eqref{eq1.5.1} with $a^{i j} = \delta_{i j}$, $b=0$, $c=0$, and $\lambda = 1$ (see Remark \ref{remark 1.8}), we have
\begin{equation*}
\begin{aligned}
 &   [u]_{ \kC^{\alpha}  (\bR^{1+2d}_T) } + [D_v u]_{ \kC^{\alpha}  (\bR^{1+2d}_T) }
 + [D^2_v u]_{ \kC^{\alpha}  (\bR^{1+2d}_T) }\\
 &\quad +\text{ sup}_{ (t, v) \in \bR^{1+d}_T }  \|u (t, \cdot, v)\|_{ C^{(2+\alpha)/3} (\bR^d) }  \\
 & \le N  (\|\partial_t u - v \cdot D_x u\| + \|u\| + \|\Delta_v u\|),
  \end{aligned}
 \end{equation*}
  where $\|\cdot\|$ stands for the $L_{\infty}  C^{\alpha/3, \alpha}_{x, v} (\bR^{1+2d}_T)$ norm.

By  interpolating between $C^{(2+\alpha)/3}_x$ and  $C_b$ and between $C^{2+\alpha}_v$ and $C_b$, we may replace the last two terms on the right-hand side of the last inequality  with
$$
    N [D^2_v u]_{L_{\infty}  C^{\alpha/3, \alpha}_{x, v} (\bR^{1+2d}_T)} +  N \|u\|_{L_{\infty}  (\bR^{1+2d}_T)}.
$$
The corollary is proved.
   \end{proof}

\begin{proof}[Proof of Corollary \ref{corollary 1.8}]
$(i)$ To prove the assertion $(i)$,   we apply Theorem \ref{theorem 1.5} $(ii)$ with $a^{i j} =\delta_{i j}$, $b=0$, $c=0$, and $\lambda \to 0$ (see Remark \ref{remark 1.8}).

$(ii)$ Let  $\phi \in C^{\infty}_0 (\widetilde Q_{ (r+R)/2 })$ be a function such that $\phi = 1$ on $Q_{r}$.
Then, by the first assertion, $u \phi \in \kC^{2, \alpha} (\bR^{1+2d}_T)$. Hence, by \eqref{eq1.5.2}  and the product rule inequality (see Remark \ref{remark 1.4}), we have
\begin{align*}
 &
   [D^2_v u]_{\kC^{\alpha} (Q_{ r })}
    \le N  [(\partial_t - v \cdot D_x) (u \phi)]_{ L_{\infty} C^{\alpha/3, \alpha}_{x, v} (\bR^{1+2d}_T)}\\
    & + N [\Delta_v (u \phi)]_{ L_{\infty} C^{\alpha/3, \alpha}_{x, v} (\bR^{1+2d}_T)}
    +N \|u \phi\|_{ L_{\infty} (\bR^{1+2d}_T)}\\
 &
    \le  N (\|\partial_t u - v \cdot D_x u\|
    +   \|u\| +  \|D_v u\| + \|D^2_v u\|) = N \|u\|_{ \mathbb{C}^{2, \alpha} (Q_2) },
\end{align*}
where by $\|\cdot\|$ we mean the $L_{\infty} C^{\alpha/3, \alpha}_{x, v} (Q_2)$ norm.
\end{proof}

\begin{proof}[Proof of Corollary \ref{corollary 1.7}]
  The proof is standard (cf. Theorem 7.1.1 in \cite{Kr_96}).
  Let $\xi \in C^{\infty}_{\text{loc}} (\bR)$ be a function such that
$\xi = 0$ if $t \geq 1$, and $\xi = 1$ if $t \leq 0$.
We denote
\begin{align*}
& f = P u + b \cdot D_v u + c u, \\
	&r_0 = r, \quad r_n  = r + (R - r) \sum_{k = 1}^n 2^{-k}, n \ge 1,\\
	&\zeta_n (t, v) = \xi \big(2^{2(n+1)} (R - r)^{-2} (- r^2_{n} - t)\big)   \, \,
	   \xi \big(2^{(n+1)} (R - r)^{-1} (|v| -  r_{n})\big)\\
&	   \quad \quad \quad \quad \times \xi \big(2^{3(n+1)}  (R-r)^{-3} (|x| -     r^{3}_{n} )\big),
\end{align*}
and note that
 $\zeta_n$ is a smooth function such that $\zeta_n = 1$ on $Q_{ r_n}$, and $\zeta_n = 0$ on $\bR^{1+2d}_0 \cap Q^c_{ r_{n+1}}$.

Next, $u \zeta_n$ satisfies the identity
$$
	(P  + b \cdot D_v  + c  +\lambda^2)  (u \zeta_n)  = f \zeta_n
	+ u (P \zeta_n + b \cdot D_v \zeta_n)
	- 2 (a D_{v} u) \cdot D \zeta_n
	+ \lambda^2 u \zeta_n.
$$
Then, by Theorem \ref{theorem 1.5} $(ii)$, for any $\lambda \ge \lambda_0$,
\begin{align}
 \label{eq1.7.4}
&
\lambda^{2} \|u \zeta_n\|_{\kC^{ \alpha} (\bR^{1+2d}_T)} + \lambda \|D_v (u \zeta_n)\|_{\kC^{ \alpha} (\bR^{1+2d}_T)}  + \|D^2_v (u \zeta_n)\|_{\kC^{\alpha}
(\bR^{1+2d}_T)} \\
&\quad + \|(\partial_t  - v \cdot D_x) (u \zeta_n)\|_{L_{\infty} C^{\alpha/3, \alpha}_{x, v} (\bR^{1+2d}_T)}
+\text{ sup}_{ (t, v) \in \bR^{1+d}_T }  \|u \zeta_n (t, \cdot, v)\|_{ C^{(2+\alpha)/3} (\bR^d) }  \notag\\
&\le N \delta^{-\theta} \lambda^{\alpha}\sum_{k = 1}^4 I_k,\notag
\end{align}
where 
$$
    I_1 = \|f \zeta_n\|, \,\, I_2 = \|u (P \zeta_n + b \cdot D_v \zeta_n)\|,\,\,I_3 = \|(a D_v u) \cdot D \zeta_n\|,  \, \,
    I_4 = \lambda^2 \|u \zeta_n\|,
$$
and
$\|\cdot\|$ is the $L_{\infty} C^{\alpha/3, \alpha}_{x, v} (\bR^{1+2d}_T)$ norm.

We now estimate the terms $I_k, k = 1 -4$.
In the sequel, $N = N (d, \alpha, K, L, r, R).$
By the product rule inequality (cf. Remark \ref{remark 1.4}),
\begin{equation*}
\begin{aligned}
    I_1  &\le N \|f\|_{L_{\infty} C^{\alpha/3, \alpha}_{x, v} (Q_{ r_{n+1}}) } \|\zeta_n\|_{L_{\infty} C^{\alpha/3, \alpha}_{x, v}  (\bR^{1+2d}_0)  }\\
  &\le N  2^{n \alpha} \|f\|_{L_{\infty} C^{\alpha/3, \alpha}_{x, v} (Q_{ r_{n+1}}) }.    \end{aligned}
\end{equation*}
Arguing as above and using Assumptions \ref{assumption 1.1} - \ref{assumption 1.3} give
\begin{align*}
&   
I_2 \le N \delta^{-1}   2^{(3 + \alpha) n}  \|u\|_{L_{\infty} C^{\alpha/3, \alpha}_{x, v} (Q_{ r_{n+1}}) },\\
&
I_3  \le N \delta^{-1} 2^{(1+\alpha) n}  \|D_v u\|_{L_{\infty} C^{\alpha/3, \alpha}_{x, v} (Q_{ r_{n+1}}) },\\
& I_4 \le N 2^{\alpha n}  \lambda^2  \|u\|_{L_{\infty} C^{\alpha/3, \alpha}_{x, v} (Q_{  r_{n+1} }) }.
 \end{align*}

We now denote
$$
         A_n =  \|D^2_v u\|_{ \kC^{\alpha} (Q_{r_n}) },
    \quad  B_n = \sup_{t, v \in (-r_n^2, 0) \times B_{r_n}} \|u (t, \cdot, v)\|_{ C^{(2+\alpha)/3} (Q_{r_n}),}
$$
$$
    C_n = \|D_v u\|_{\kC^{\alpha} (Q_{ r_{n}})}
$$
and we set
 \begin{equation}
                    \label{eq1.7.6}
    \lambda = 2^{ \beta n}  \lambda_1 (d, \alpha, K, L, \delta, r) \ge \lambda_0,
  \end{equation}
 where $\beta  > 2$
 and $\lambda_1 > 1$ will  be determined later.
Note that $4\beta>\max\{\alpha\beta+3+\alpha,\alpha+(2+\alpha)\beta\}$.
Combining \eqref{eq1.7.4} - \eqref{eq1.7.6}  gives
\begin{align}
                \label{eq1.7.8}
&   [u]_{ \kC^{\alpha} (Q_{ r}) } + \|\partial_t u - v \cdot D_x u\|_{L_{\infty} C^{\alpha/3, \alpha}_{x, v} (Q_{ r}) }   + A_n+ B_n + \lambda_1  2^{\beta n}  C_n\\
 &\le N \delta^{-\theta} \big(\lambda_1^{\alpha} 2^{(\alpha + \alpha \beta) n} \|f\|_{L_{\infty} C^{\alpha/3, \alpha}_{x, v} (Q_{R}) }
  + \lambda_{ 1 }^{2+\alpha} 2^{  4 \beta n} \|u\|_{L_{\infty} C^{\alpha/3, \alpha}_{x, v} (Q_{ r_{n+1} })}\notag\\
  &
  \quad + \lambda_{1}^{\alpha}  2^{(1+\alpha+\alpha \beta) n}  C_{n+1}\big)\notag.
 \end{align}
 By the standard interpolation inequality (see Lemma \ref{lemma B.3}) and a scaling argument,
 \begin{align*}
    \|u\|_{L_{\infty} C^{\alpha/3, \alpha}_{x, v} (Q_{ r_{n+1}}) } \le    N  \varepsilon^{2} (A_{n+1} + B_{n+1})
     +  N \varepsilon^{-\alpha} \|u\|_{ L_{\infty} (Q_{r_{n+1}}) }.
 \end{align*}

We take
 $$
    \varepsilon = \varepsilon_0 \lambda_1^{-(2+\alpha)/2 } (N \delta^{-\theta})^{-1/2} 2^{-2 \beta n },
 \quad\beta > \max\bigg\{\frac{1+\alpha}{1-\alpha}, 2\bigg\},
 $$
 where $\varepsilon_0 \in (0, 1)$ will be determined later,
 so that
 \begin{equation}
            \label{eq1.7.9}
 \begin{aligned}
  &  N \delta^{-\theta} \lambda_1^{2+\alpha} 2^{4 \beta n}  \|u\|_{L_{\infty} C^{\alpha/3, \alpha}_{x, v} (Q_{ r_{n+1}}) }\\
   & \le \varepsilon_0^2 (A_{n+1} + B_{n+1}) + N \varepsilon^{-\alpha}_0 \delta^{-\theta} \lambda_1^{(2+\alpha)(1+\alpha/2)}   2^{ (4 + 2 \alpha) \beta n} \|u\|_{ L_{\infty} (Q_R), }
 \end{aligned}
 \end{equation}
  \begin{equation}
 \label{eq1.7.7}
    \beta > 1 +\alpha + \alpha \beta.
 \end{equation}
We multiply both sides of \eqref{eq1.7.8} by $2^{- 6\beta n}$ and  sum over $n \in  \{0, 1, 2, \ldots\}$. Due to \eqref{eq1.7.9} -  \eqref{eq1.7.7}, we get
\begin{equation}
                \label{eq1.7.10}
\begin{aligned}
&  [u]_{ \kC^{\alpha} (Q_{r}) }
+ \|\partial_t u - v \cdot D_x u\|_{L_{\infty} C^{\alpha/3, \alpha}_{x, v} (Q_{ r}) } \\
& \quad + \sum_{n = 0}^{\infty} 2^{- 6\beta  n}  (A_n + B_n)   + \lambda_1 \sum_{n = 0}^{\infty} 2^{- 5 \beta n}  C_n  \\
&\le N  \delta^{-\theta} \big( \lambda_1^{\alpha}\|f\|_{L_{\infty} C^{\alpha/3, \alpha}_{x, v} (Q_{R}) }
  +  \varepsilon^{-\alpha}_0 \lambda_1^{ (2+\alpha)(1+\alpha/2)}\|u\|_{L_{\infty} C^{\alpha/3, \alpha}_{x, v} (Q_R)}\big)\\
  &
  \quad + \varepsilon^2_0 2^{6 \beta}  \sum_{n = 1}^{\infty} 2^{- 6 \beta n}
  (A_{n} + B_n) + N \delta^{-\theta}  2^{ 5 \beta}   \lambda_1^{\alpha} \sum_{n = 1}^{\infty} 2^{ - 5 \beta n} C_{n}.
 \end{aligned}
\end{equation}
  Taking $\lambda_1  > 1$ large   so that
 $$
    \lambda_1 -  2^{  5 \beta} N \delta^{-\theta} \lambda_1^{\alpha}  > \lambda_1/2\quad \text{and} \quad \varepsilon_0 = 2^{-3 \beta - 1},
 $$
 we may drop the last two terms on the right-hand side of \eqref{eq1.7.10}. The assertion is proved.
\end{proof}


\appendix
\section{\texorpdfstring{$S_2$}{S2} regularity results  for the KFP equations}

\begin{definition}
            \label{definition A.5}
We say that $u \in S_2 (\bR^{1+2d}_T)$ is a solution to \eqref{1.1} if the identity
\begin{equation}
            \label{eqA.5.1}
    \partial_t u - v \cdot D_x u = a^{i j} D_{v_i v_j} u - b \cdot D_v u - (c + \lambda^2) u
\end{equation}
holds in $L_2 (\bR^{1+2d}_T)$.
Furthermore, for finite $S < T$,  $u \in  S_2 ((S, T) \times \bR^{2d})$ is a solution to the Cauchy problem
\eqref{eq1.5.4} if \eqref{eqA.5.1} holds in $L_2 ((S, T) \times \bR^{2d})$ with $\lambda = 0$, and
there exists $U \in S_2 (\bR^{1+2d}_T)$ such that
$U \equiv u$ on $(S, T) \times \bR^{2d}$,  $U \equiv 0$ on $(-\infty, T) \times \bR^{2d}$.
\end{definition}

\begin{theorem}[see Theorem 2.6 of \cite{DY_21a}]
			\label{theorem A.4}
Let $\alpha \in (0, 1]$, $a$ be a function satisfying Assumptions \ref{assumption 1.1} - \ref{assumption 1.2} and $b, c \in L_{\infty} (\bR^{1+2d}_T)$. Then, there  exists  $\lambda_0 > 1$  as in \eqref{eq1.5.0}
 such that for any $\lambda \ge \lambda_0$ and
$f \in  L_2 (\bR^{1+2d}_T)$,
Eq. \eqref{1.1} has a unique solution $u \in S_2 (\bR^{1+2d}_T)$.
\end{theorem}

\begin{theorem}[see Theorem 4.1 of \cite{DY_21a}]
	\label{theorem A.1}
Let $a = a^{i j} (t)$ be a function satisfying Assumption \ref{assumption 1.1} and recall the notation \eqref{eq3.0}.
Then, the following assertions hold.

$(i)$ For any $\lambda \ge 0$ and $u \in S_2 (\bR^{1+2d}_T)$,
\begin{equation*}
    \begin{aligned}
&	\lambda^2 \|u \| + \lambda  \|D_v u\|+\| D^2_v u \| +  \|(- \Delta_x)^{1/3} u \|
 + \|D_v (-\Delta_x)^{1/6} u\| \\
 &\leq  \delta^{-1} \| P_0  u + \lambda^2 u\|,
\end{aligned}	
\end{equation*}
where $\|\cdot\|=\|\cdot\|_{ L_2  (\bR^{1+2d}_T) }$.
Furthermore, for any $\lambda \ne 0$, the equation
$$
    (P_0  + \lambda^2) u = f
$$
has a unique solution $u \in S_2 (\bR^{1+2d}_T)$.

$(iii)$
For any finite numbers $S < T$ and $f \in L_2 ((S, T) \times \bR^{2d})$,
the Cauchy problem \eqref{eq1.5.4} with $P = P_0$, $b \equiv 0$, and $c \equiv 0$ has a unique solution
$u \in
S_2((S, T) \times \bR^{2d})$.
In addition,
\begin{align*}
&    \||u|+ |D_v u|+|D^2_v u|+|(-\Delta_x)^{1/3}  u| + |D_v (-\Delta_x)^{1/6}  u|+|\partial_t u - v \cdot D_x u|\|\\
&  \leq N (d, T-S) \delta^{-1} \|f\|,
\end{align*}
where $\|\cdot\| = \|\cdot\|_{ L_2 ((S, T)\times \bR^{2d}) }$.
\end{theorem}

\begin{corollary}
\label{corollary A.4}
For any $u \in S_2 (\bR^{1+2d}_T)$, we have $(-\Delta_x)^{1/3} u \in L_2 (\bR^{1+2d}_T)$.
\end{corollary}

\begin{proof}
Let $f = \partial_t u - v \cdot D_x u - \Delta_v u \in L_2 (\bR^{1+2d}_T)$. Applying Theorem \ref{theorem A.1} with $a^{i j}  \equiv \delta_{i j}$, we prove the desired assertion.
\end{proof}

\begin{corollary}
                \label{corollary A.2}
Invoke the assumption of Theorem \ref{theorem A.4}  and assume, additionally, that
$$
    D^n_v D^m_x h \in C_b (\overline{\bR^{1+2d}_T}), \quad \forall n, m \ge 0, \quad h = a, b, c, f,
$$
and $D^n_v D^m_x f \in L_2 (\bR^{1+2d}_T),\,\, \forall n, m \ge 0$.
Then,  $D^n_v D^m_x u \in  C_b (\overline{\bR^{1+2d}_T}) \cap L_2 (\bR^{1+2d}_T)$ for  $n, m \ge 0$.
\end{corollary}

\begin{proof}
To make the argument presented below rigorous, one needs to use the method of finite-difference quotients.
By using an induction argument similar to that used in the proof of Lemma \ref{lemma 3.6}, one can show that for any multi-indexes $\alpha$ and $\beta$, and $U = D_v^{\alpha} D_x^{\beta} u \in S_2 (\bR^{1+2d}_T)$, so that
$$
    (P + b \cdot D_v + c + \lambda^2) U  =: F \in L_2 (\bR^{1+2d}_T).
$$
We multiply the above identity by $U$, integrate over $\bR^{1+2d}_s$, and note that the term containing $v \cdot D_x |U|^2$ vanishes. We conclude that
$$
    \int_{\bR^{2d}} U^2 (s, x, v) \, dx dv < \infty \quad \text{a.e.} \, \, s \in (-\infty, T).
$$
An application of the Sobolev embedding theorem finishes the proof of this assertion.
\end{proof}

\begin{lemma}[Interior $S_2$ estimate, see Lemma 4.5 in \cite{DY_21a}]
				\label{lemma A.3}
Let $a = a (t)$ satisfy Assumption \ref{assumption 1.1},  $\lambda \in \bR$, and $0 < r < R$ be numbers.
Then, for any $u \in S_{2,\text{loc}} (\bR^{1+2d}_0)$,
\begin{equation*}
\begin{aligned}
& \,
 	 \|\partial_t u -v \cdot D_x u\|_{ L_2 ( Q_r) } + \delta^{-2 }(r_2 - r_1)^{-1} \|D_v u\|_{ L_2 ( Q_r) }+ 	\| D^2_v u\|_{ L_2 ( Q_r) }\\
 	 &
	\leq N (d) \delta^{-1}  \|P_0 u + \lambda^2 u\|_{ L_2 ( Q_R) }+  N (d) \delta^{-4} R (R-r)^{-3}  \|u\|_{ L_2 ( Q_R) },
\end{aligned}	
\end{equation*}
where $P_0$ is defined by \eqref{eq3.0}.
\end{lemma}

\section{}

\begin{lemma}[Lemma 3.1 in \cite{DY_21a}]
						\label{lemma B.1}
Let $r > 0$ be  a number.
Then, the following assertions hold.

$(i)$ For any $z, z_0 \in \bR^{1+2d}$,
$$
	\rho (z, z_0) \leq 2 \rho (z_0, z).
$$

$(ii)$
For any $z, z_0, z_1 \in \bR^{1+2d}$,
$$
	\rho (z, z_0)
	\leq 2 (\rho (z, z_1) + \rho (z_1, z_0)).
$$

$(iii)$
The function $\widehat \rho$ (see \eqref{1.4}) is a (symmetric) quasi-distance.

$(iv)$ One has
	$$
			 \widehat{Q}_{r} (z_0) \subset  \tQ_{r} (z_0) \subset   \widehat{Q}_{ 3r} (z_0),
	$$
	where $\widetilde Q_r (z_0)$ and $\widehat Q_r (z_0)$ are defined in \eqref{eq1.13} and \eqref{eq1.14}, respectively.

$(v)$ For $T \in (-\infty, \infty]$,
$$
    \frac{|\widehat Q_{2r} (z_0) \cap \bR^{1+2d}_T|}{|\widehat Q_r (z_0) \cap \bR^{1+2d}_T|} \le N (d),
$$
so that the triple $(\overline{\bR^{1+2d}_T}, \widehat \rho, dz)$ (with the induced topology if $T < \infty$) is a space of homogeneous type.
\end{lemma}

For the proof of the following inequality see, for instance,  Lemma 6.3.1 in \cite{Kr_96}.
\begin{lemma}[Standard interpolation inequality in H\"older spaces]
            \label{lemma B.3}
Let $\Omega$ be either $\bR^d$ or a bounded domain with a smooth boundary, $u \in C^{k+\alpha} (\Omega)$, $k \in \{0, 1, \ldots\}, \alpha \in [0, 1]$ be the usual H\"older space.
Then, for any $j = 0, 1, \ldots, k$, and $\beta \in [0, 1]$ such that $j+\beta < k+\alpha$ and any $\varepsilon > 0$, one has
$$
    [D^j u]_{ C^{ \beta} ( \Omega) } \le N (\varepsilon^{k+\alpha - j-\beta} [u]_{ C^{k+\alpha} (\Omega) } +  (1+\varepsilon^{-j-\beta})  \|u\|_{ L_{\infty} ( \Omega) }).
$$
where $N = N (d, k, \alpha, j, \beta, \Omega)$. In the case when $\Omega = \bR^d$, one can replace the factor $1+\varepsilon^{-j-\beta}$ with $\varepsilon^{-j-\beta}$ on the right-hand side of the above inequality.
\end{lemma}

\begin{lemma}
        \label{lemma 5.1}
Let $\alpha \in (0, 1]$ and $u \in L_{\infty} C^{\alpha/3, \alpha}_{x, v} (\bR^{1+2d}_T)$ be a function such that $\partial_t u - v \cdot D_x u \in  L_{\infty}  (\bR^{1+2d}_T)$.
Then, $u \in \kC^{\alpha} (\bR^{1+2d}_T)$, and furthermore,  for any $\varepsilon > 0$, one has
\begin{equation*}
\begin{aligned}
 [u]_{\kC^{\alpha} (\bR^{1+2d}_T)} \le  [u]_{L_{\infty} C^{\alpha/3, \alpha}_{x, v} (\bR^{1+2d}_T)} +\varepsilon^{2-\alpha} \|\partial_t u - v \cdot D_x u\|_{L_{\infty}  (\bR^{1+2d}_T)}.
\end{aligned}
\end{equation*}
\end{lemma}

\begin{proof}
Note that by using Lemma \ref{lemma 2.0} and a scaling argument, we only need to prove the assertions with $\varepsilon = 1$.
We will show that
\begin{equation}
            \label{eq5.2}
    |u (z_0) - u (z)|
    \le \big(2[u]_{L_{\infty} C^{\alpha/3, \alpha}_{x, v} (\bR^{1+2d}_T)} + \||u| + |f|\|_{L_{\infty} (\bR^{1+2d}_T}\big) \rho^{\alpha} (z_0, z).
\end{equation}
Fix $z_0, z \in \bR^{1+2d}_T$.
By shifting (see Lemma \ref{lemma 2.0}), we may also assume that $z_0 = 0$.
We denote $\partial_t u - v \cdot D_x u = f$. Then, by the fundamental theorem of calculus,
$$
    u (z) = u (0, x + tv, v) + \int_0^t f (t', x + (t-t') v, v) \, dt'.
$$
 We then obtain
\begin{align*}
  &  |u (0) - u (z)| \le |u (0, x + tv, v) - u (0)| + \int_0^t |f (t', x + (t-t')v, v)|\, dt'\\
    &\le [u]_{L_{\infty} C^{\alpha/3, \alpha}_{x, v} (\bR^{1+2d}_T)} (|x + t v|^{1/3} +|v|)^{\alpha} + t^{\alpha/2}  \|f\|_{L_{\infty} (\bR^{1+2d}_T)}\\
&
    \le \big(2[u]_{L_{\infty} C^{\alpha/3, \alpha}_{x, v} (\bR^{1+2d}_T)} + \|f\|_{L_{\infty} (\bR^{1+2d}_T)}\big) \rho^{\alpha} (0, z),
\end{align*}
and, thus, \eqref{eq5.2} is valid.
\end{proof}

\begin{lemma}
        \label{lemma B.4}
Let $\alpha \in (0, 1)$, $c \ge 1$, $r > 0$, $z_0 \in \overline{\bR^{1+2d}_T}$, $f$ and $h$ be measurable functions such that $[f]_{L_{\infty}  C^{\alpha/3, \alpha}_{x, v} (\bR^{1+2d}_T)}$, $[h]_{\kC^{\alpha} (\bR^{1+2d}_T)} < \infty$, and $\chi (t) := f (t, x_0 - (t-t_0)v_0, v_0)$.
Then, the following assertions hold.
\begin{align}
\label{eqB.4.1}
  &(i)  \sum_{k =0}^{\infty} 2^{-k} (|f - \chi|^2)_{Q_{r, 2^k c  r}(z_0)}^{1/2} \le N [f]_{ L_{\infty} C^{\alpha/3, \alpha}_{x, v} (\bR^{1+2d}_T) } (c r)^{\alpha},\\
  &(ii) \sum_{k =0}^{\infty} 2^{-k} (|h - (h)_{Q_{r, 2^k c  r }(z_0)}|^2)_{Q_{r, 2^k c  r }(z_0)}^{1/2}
  \le N  [h]_{\kC^{\alpha} (\bR^{1+2d}_T) } (c r)^{\alpha}\notag,
\end{align}
where $N = N (\alpha)$.
\end{lemma}

\begin{proof}
$(i)$ Denote $A = [f]_{L_{\infty} C^{\alpha/3, \alpha}_{x, v} (\bR^{1+2d}_T)}$ and note that
for any $z \in Q_{r, 2^k c r} (z_0)$, we have
\begin{align*}
     &|f (t, x, v) - f (t, x_0 - (t-t_0) v_0, v_0)|  \le A (|x-x_0+(t-t_0)v_0|^{1/3} + |v-v_0|)^{\alpha} \\
     &\le A (2^k c)^{\alpha}   ((2^k c)^{-1}|x-x_0+(t-t_0)v_0|^{1/3} +  |v-v_0|)^{\alpha} \le N A (2^k c)^{\alpha} r^{\alpha}.
\end{align*}
Then, the series on the left-hand side of \eqref{eqB.4.1} is less then
$$
    N A (c r)^{\alpha} \sum_{k=0}^{\infty} 2^{(-1+\alpha)k} \le N A (c r)^{\alpha},
$$
and hence, \eqref{eqB.4.1} is true.

$(ii)$ For any $z_1, z_2$ such that $z_i \in Q_{r, 2^k c r} (z_0), i = 1, 2$, by Lemma \ref{lemma B.1} $(i)$ and $(ii)$, one has
$$
    |h (z_1) - h (z_2)| \le  N  [h]_{\kC^{\alpha} (\bR^{1+2d}_T) } (\rho^{\alpha} (z_1, z_0) + \rho^{\alpha} (z_2, z_0))
    \le N [h]_{\kC^{\alpha} (\bR^{1+2d}_T) } (2^k c r)^{\alpha}.
$$
The last inequality and the fact that
$$
    (|h - (h)_G|^2)_{G} \le  \fint_{G} \fint_G |h (z_1) - h (z_2)|^2 \, dz_1 dz_2
$$
imply the validity of the assertion $(ii)$.
\end{proof}

\begin{lemma}
        \label{lemma B.5}
Let $s \in (0, 1/2)$.

$(i)$ For any Schwartz function $u$, the following pointwise formula holds:
\begin{equation*}
        D_x (-\Delta_x)^{-s} u (x) = N (d, s) \, \text{p.v.} \int u (x-y) \frac{y}{|y|^{d-2s+2}} \, dy.
\end{equation*}
This formula is also valid for  $u \in C^1_0 (\bR^d)$.

$(ii)$ For any $u \in C^2_0 (\bR^d)$, one has
\begin{equation*}
    \big(D_x (-\Delta_x)^{-s}\big) \big((-\Delta_x)^{s} u\big) \equiv D_x u.
\end{equation*}
\end{lemma}


\begin{thebibliography}{m}

\bibitem{AP_20}
F. Anceschi, S. Polidoro,
 A survey on the classical theory for Kolmogorov equation. Matematiche (Catania) 75 (2020), no. 1, 221--258.


\bibitem{AV_04} Radjesvarane Alexandre, C\'edric Villani,  On the Landau approximation in plasma physics. Ann. Inst. H. Poincar\'e Anal, Non Lin\'eaire 21 (2004), no. 1, 61--95.

\bibitem{BB_22} Stefano Biagi, Marco Bramanti,
Schauder estimates for Kolmogorov-Fokker-Planck operators with coefficients measurable in time and H\"older continuous in space, arXiv:2205.10270v2



\bibitem{Boc_13} Serena Boccia,  Schauder estimates for solutions of higher-order parabolic systems, Methods Appl. Anal. 20 (2013), no. 1, 47--67.


\bibitem{B_02} Fran\c{c}ois Bouchut,  Hypoelliptic regularity in kinetic equations, J. Math. Pures Appl. (9) 81 (2002), no. 11, 1135--1159


\bibitem{BCLP_13} Marco Bramanti, Giovanni Cupini, Ermanno Lanconelli, Enrico Priola,
 Global $L^p$ estimates for degenerate Ornstein-Uhlenbeck operators with variable coefficients. Math. Nachr. 286 (2013), no. 11-12, 1087--1101.

\bibitem{B_69} A. Brandt, Interior Schauder estimates for parabolic differential- (or difference-) equations via the maximum principle. Israel J. Math. 7 (1969), 254--262.


\bibitem{C_76} Alberto P. Calder\'on, Inequalities for the maximal function relative to a metric. Studia Math.,
57(3):297-306, 1976.


\bibitem{CHM_21} Paul-\'Eric Chaudru de Raynal,  Igor Honor\'e,  St\'ephane Menozzi,  Sharp Schauder estimates for some degenerate Kolmogorov equations, Ann. Sc. Norm. Super. Pisa Cl. Sci. (5) 22 (2021), no. 3, 989--1089.



\bibitem{CZ_19} Zhen-Qing  Chen,  Xicheng Zhang, Propagation of regularity in $L^p$-spaces for Kolmogorov-type hypoelliptic operators. J. Evol. Equ. 19 (2019), no. 4, 1041--1069.


\bibitem{DFP_06} Marco Di Francesco, Sergio Polidoro,  Schauder estimates, Harnack inequality and Gaussian lower bound for Kolmogorov-type operators in non-divergence form, Adv. Differential Equations 11 (2006), no. 11, 1261--1320.


\bibitem{DGO_20} Hongjie Dong, Yan Guo, Zhimeng Ouyang, The Vlasov-Poisson-Landau System with the Specular-Reflection Boundary Condition,  arXiv:2010.05314v3,  Arch. Ration. Mech. Anal., to appear (2022).


\bibitem{DGY_21} Hongjie Dong, Yan Guo, Timur Yastrzhembskiy, Kinetic Fokker-Planck and Landau equations with specular reflection boundary condition, Kinetic \& Related Models  15 (2022), no. 3, 467--516.


\bibitem{DK_13} Hongjie Dong, Doyoon Kim,  Schauder estimates for a class of non-local elliptic equations. Discrete Contin. Dyn. Syst. 33 (2013), no. 6, 2319--2347.

\bibitem{DK_11} Hongjie Dong,  Seick Kim,  Partial Schauder estimates for second-order elliptic and parabolic equations, Calc. Var. Partial Differential Equations 40 (2011), no. 3-4, 481--500.

\bibitem{DK_19} Hongjie Dong,  Seick Kim,  Partial Schauder estimates for second-order elliptic and parabolic equations: a revisit, Int. Math. Res. Not. IMRN 2019, no. 7, 2085--2136.


\bibitem{DY_21a} Hongjie Dong, Timur Yastrzhembskiy,
Global $L_p$ estimates for kinetic Kolmogorov-Fokker-Planck equations in nondivergence form,   Arch. Ration. Mech. Anal.  245 (2022), no. 1, 501--564




\bibitem{DY_21b} Hongjie Dong, Timur Yastrzhembskiy,
Global $L_p$ estimates for kinetic Kolmogorov-Fokker-Planck equations in divergence form, arXiv:2206.03370v2



\bibitem{DZ_15} Hongjie Dong, Hong Zhang,  Schauder estimates for higher-order parabolic systems with time irregular coefficients. Calc. Var. Partial Differential Equations 54 (2015), no. 1, 47--74.



\bibitem{GM_12}
Mariano Giaquinta, Luca Martinazzi, An introduction to the regularity theory for elliptic systems, harmonic maps and minimal graphs. Second edition, Appunti. Scuola Normale Superiore di Pisa (Nuova Serie) [Lecture Notes. Scuola Normale Superiore di Pisa (New Series)], 11. Edizioni della Normale, Pisa, 2012. xiv+366 pp.


\bibitem{Fi_63} Paul Fife,  Schauder estimates under incomplete H\"older continuity assumptions, Pacific J. Math. 13 (1963), 511--550.

\bibitem{F_64} Avner Friedman,  Partial differential equations of parabolic type, Prentice-Hall, Inc., Englewood Cliffs, N.J. 1964 xiv+347 pp.

\bibitem{HMZ_20} Zimo Hao, Mingyan Wu, Xicheng Zhang,  Schauder estimates for nonlocal kinetic equations and applications, J. Math. Pures Appl. (9) 140 (2020), 139--184.

\bibitem{HS_20}  Christopher Henderson, Stanley Snelson, $C^{\infty}$   smoothing for weak solutions of the inhomogeneous Landau equation, Arch. Ration. Mech. Anal. 236 (2020), no. 1, 113--143.


\bibitem{HW_22} Christopher Henderson, Weinan Wang, Kinetic Schauder estimates with time-irregular coefficients and uniqueness for the Landau equation, arXiv:2205.12930v2

\bibitem{IS_21} Cyril Imbert, Luis Silvestre,  The Schauder estimate for kinetic integral equations, Anal. PDE 14 (2021), no. 1, 171--204.

\bibitem{IM_21} Cyril Imbert, Cl\'ement Mouhot,  The Schauder estimate in kinetic theory with application to a toy nonlinear model, Ann. H. Lebesgue 4 (2021), 369--405.



\bibitem{Kn_81} Barry F. Knerr,  Parabolic interior Schauder estimates by the maximum principle, Arch. Rational Mech. Anal. 75 (1980/81), no. 1, 51--58.


\bibitem{Kr_96} Nicolai V. Krylov, Lectures on elliptic and parabolic equations in H\"older spaces. Graduate Studies in Mathematics, 12. American Mathematical Society, Providence, RI, 1996. xii+164 pp

\bibitem{KrP_10} Nicolai V. Krylov, Enrico Priola,  Elliptic and parabolic second-order PDEs with growing coefficients, Comm. Partial Differential Equations 35 (2010), no. 1, 1--22.

\bibitem{Kr_08} Nicolai V. Krylov,  Lectures on elliptic and parabolic equations in Sobolev spaces. Graduate Studies in Mathematics, 96. American Mathematical Society, Providence, RI, 2008.


\bibitem{Li_92} Gary M. Lieberman,  Intermediate Schauder theory for second order parabolic equations. IV. Time irregularity and regularity, Differential Integral Equations 5 (1992), no. 6, 1219--1236.

\bibitem{Lo_00} Luca Lorenzi,  Optimal Schauder estimates for parabolic problems with data measurable with respect to time, SIAM J. Math. Anal. 32 (2000), no. 3, 588--615.

\bibitem{L_97} Alessandra  Lunardi, Schauder estimates for a class of degenerate elliptic and parabolic operators with unbounded coefficients in $R^n$. Ann. Scuola Norm. Sup. Pisa Cl. Sci. (4) 24 (1997), no. 1, 133--164.


\bibitem{M_97} Maria Manfredini,  The Dirichlet problem for a class of ultraparabolic equations, Adv. Differential Equations 2 (1997), no. 5, 831--866.

\bibitem{P_05} Andrea Pascucci, Kolmogorov equations in physics and in finance, Elliptic and parabolic problems, Progr. Nonlinear Differential Equations Appl., 63, Birkh\"auser, Basel, 2005, 353--364.




\bibitem{P_14} Grigorios A. Pavliotis,  Stochastic processes and applications. Diffusion processes, the Fokker-Planck and Langevin equations. Texts in Applied Mathematics, 60. Springer, New York, 2014. xiv+339 pp.


\bibitem{PRS_22} Sergio Polidoro, Annalaura Rebucci, Bianca Stroffolini,  Schauder type estimates for degenerate Kolmogorov equations with Dini continuous coefficients, Commun. Pure Appl. Anal. 21 (2022), no. 4, 1385--1416.



\bibitem{P_09} Enrico Priola,  Global Schauder estimates for a class of degenerate Kolmogorov equations. Studia Math. 194 (2009), no. 2, 117--153.

\bibitem{Sch_96} Wilhelm Schlag,  Schauder and $L^p$ estimates for parabolic systems via Campanato spaces, Comm. Partial Differential Equations 21 (1996), no. 7--8, 1141--1175.

 \bibitem{S_97} Leon Simon, Schauder estimates by scaling, Calc. Var. Partial Differential Equations 5 (1997), no. 5, 391--407.


\bibitem{TW_10} Guji Tian, Xu-Jia Wang,  Partial regularity for elliptic equations, Discrete Contin. Dyn. Syst. 28 (2010), no. 3, 899--913.

\bibitem{T_86} Neil S. Trudinger,  A new approach to the Schauder estimates for linear elliptic equations. Miniconference on operator theory and partial differential equations (North Ryde, 1986), 52--59, Proc. Centre Math. Anal. Austral. Nat. Univ., 14, Austral. Nat. Univ., Canberra, 1986.

\end{thebibliography}
\end{document}